\documentclass[12pt]{amsart}
\usepackage{epsfig,psfrag,verbatim,amssymb,amsmath,amsfonts,graphicx,bbm,stackrel,xcolor,flexisym}
\newtheorem{theorem}{Theorem}
\newtheorem{lemma}[theorem]{Lemma}
\newtheorem{prop}[theorem]{Proposition}
\newtheorem{cor}[theorem]{Corollary}
\newtheorem{ex}[theorem]{Example}

\newtheorem{cex}[theorem]{Counterexample}
\newtheorem{remark}[theorem]{Remark}

\newtheorem{defi}[theorem]{Definition}
\newtheorem{problem}[theorem]{Problem}

\renewcommand{\phi}{\varphi}

\newcommand{\R}{\mathbb{R}}

\newcommand{\Z}{\mathbb{Z}}
\newcommand{\N}{\mathbb{N}}

\newcommand{\wP}{\widetilde{P}}

\newcommand{\vP}{\widetilde{P}_v}

\newcommand{\wE}{\widetilde{E}}

\renewcommand{\tilde}{\widetilde}
\renewcommand{\kappa}{\varkappa}

\newcommand{\F}{\mathcal{F}}

\newcommand{\si}{\sigma}

\newcommand{\Om}{\Omega}
\newcommand{\eps}{\varepsilon}

\newcommand{\vp}{\varphi}
\renewcommand{\epsilon}{\varepsilon}
\renewcommand{\mid}{\,|\,}

\newcommand{\1}{\mathbf{1}}

\newcommand{\cal}{\mathcal}

\newcounter{constante}
\newcommand{\con}[1]{
\immediate\write 1{\noexpand\newlabel{#1}{{\theconstante}{\theconstante}}}
                    c_{\theconstante}
                    \stepcounter{constante}
                   }

\begin{document}

\title[Zero-one law for frog processes] {A zero-one law for recurrence and transience of frog processes}

\author{Elena Kosygina and Martin P.W.\ Zerner} 


\begin{abstract}
We provide sufficient conditions for the validity of a dichotomy, i.e.\ zero-one law, between
  recurrence and transience of general frog models.
In particular, the results cover frog models with i.i.d.\ numbers of frogs per
  site where the frog dynamics are given by quasi-transitive Markov chains or by random walks in a common random
  environment including
  super-critical percolation clusters on $\Z^d$.  
  We also give a sufficient and almost sharp condition for recurrence of uniformly elliptic frog processes on $\Z^d$.
Its proof uses the general
  zero-one law.
\end{abstract}
\maketitle
\section{Introduction}
Frog models are  interacting particle systems which loosely
can be described as follows. The system starts with a set of inactive
particles located at points of a countably infinite state space. At
time $0$ all particles at a certain initial site get activated. 
Each
active particle moves from site to site of the state space; when it visits a site containing
inactive particles, it wakes them all up. The newly activated particles start moving in a similar
fashion.  

The first results for this type of model (under the name ``egg model'')
were published in 1999, \cite[Section 2.4]{TW99}. The vivid term ``frog
model'', where each particle is called a frog, was coined
by R. Durrett (\cite[p.\,278]{Pop03}) and became standard. For a review of
some frog models and results up to 2003 see \cite{Pop03}. Since
then, a number of papers have appeared concerning recurrence and transience
 of
frog models, see \cite{GS09}, \cite{DP14}, \cite{HJJ14},
\cite{HJJ15}. For other aspects and generalizations of the model we
refer to \cite{LMP05}, \cite{KS08}, \cite{HW15}, \cite{GNR15}, and
the references therein.

The current paper addresses zero-one laws for recurrence and
transience of frog models in a very general setting. A site is said to
be recurrent (for a precise formulation see Definition \ref{rt}) if
activating all frogs at this site causes this site to be visited by
infinitely many frogs originating at distinct sites. Otherwise it is
said to be transient.  Note that according to our definition sites
without sleeping frogs, i.e.\ inactive particles, are transient. In
the study of recurrence (resp.\ transience) of the frog model one
often constructs an event on which a given site is recurrent (resp.\
transient) with positive probability. A zero-one law allows one to
immediately conclude that a site is then recurrent (resp.\ transient)
with probability 1, and that, moreover, under natural conditions, all
sites are recurrent (resp.\ transient) with probability 1.  The proof
of Theorem~\ref{log} below illustrates such an approach. For recent
results about zero-one laws for frog models and their
applications we refer to \cite{GS09}, \cite{HJJ14}, \cite{HJJ15}.
 
We hope that our general result Theorem \ref{012} covers a
sufficiently wide range of frog models to be
useful.\footnote{However, our results do not cover models
    where the numbers of frogs per site are inhomogeneous such as, for
    example, in \cite{Pop01}. See also
    Counterexample~\ref{pop} below.} Moreover, in spite of somewhat
  bulky notation caused by generality, the idea of the proof is
  simple. 

To illustrate the range of applicability of our main theorem,
Theorem~\ref{012}, we present two of its corollaries, see
  Section~\ref{3} for additional examples. Throughout we denote by
$\eta(x)$ the number of frogs which are initially sleeping at site $x$
of a countably infinite state space $V$.  We let $S_j(x,i)$ be the position
after $j\ge 0$ steps of the $i$-th frog, $i\ge 1,$ which was sleeping
at site $x\in V$ if it is woken up. In particular, $S_0(x,i)=x$ for
all $x\in V$ and $i\ge 1$. For the first result, recall that a
stochastic matrix $K$ with state space $V$ is called {\em
  transitive}\footnote{The notion of transitive Markov chains can be
  found, e.g., in \cite[p.\ 13]{Woe00}, \cite[Section 2.6.2]{LPW09},
  \cite[Definition 10.22]{LP}, and already in \cite[Chapter X, \S
  6]{D65}.} if for any $x,y\in V$ there is a permutation $\vp$ of $V$
such that $\vp(x)=y$ and $K(\vp(u),\vp(v))=K(u,v)$ for all $u,v\in V$.
\begin{theorem}[\rm\bf  Transitive Markov chains]\label{01}
  Let the frog trajectories $(S_j(x,i))_{j\ge 0}$, $x\in V, i\ge 1$,
  be Markov chains with a common transitive and irreducible transition
  matrix on a countably infinite state space $V$. Let the numbers
  $\eta(x)$, $x\in V,$ of sleeping frogs be identically distributed
  and assume that the quantities $\eta(x), (S_j(x,i))_{j\ge 0}$,
  $x\in V, i\ge 1,$ are independent. Then either with probability 1
  every $x\in V$ is transient or with probability 1 every $x\in V$
  with $\eta(x)\ge 1$ is recurrent.
\end{theorem}
Theorem~\ref{01}, in particular, implies the
validity of Conjecture 2 in \cite[Section 3]{GS09}. Since our setting
is slightly different from the one in \cite{GS09}, we  provide a
detailed discussion in Appendix~\ref{A}.

The proof of Theorem \ref{01} is particularly simple in the case
considered in \cite[Theorem 2.3]{GS09}, where the frogs perform
independent homogeneous nearest-neighbor random walks on $\Z$ with
drift to the right. It goes as follows. At the beginning we assign at
random to each frog a trajectory which the frog will follow once it
has been woken up.  For $x\in \Z$ denote by $R_x$ the event that
waking up the frogs at $x$, when everybody else is still asleep,
results in infinitely many frogs visiting $x$.  Since
$(\eta(x))_{x\in\Z}$ is i.i.d.\ and the frogs move independently of
each other, the sequence $(\1_{R_x})_{x\in\Z}$ is stationary and ergodic
with respect to (w.r.t.) the shifts on $\Z$.  Therefore, this sequence is
either a.s.\ identically equal to 0, in which case $P[R_x]=0$ for all
$x\in\Z$, or a.s.\ $\1_{R_x}=1$ for infinitely many $x\ge 1$.  In the
latter case, if we wake up the frogs at a site $v\in \Z$ with
$\eta(v)\ge 1$, then these frogs will a.s.\ be transient to the right
and will therefore also visit a site $x$ for which $R_x$ occurs and
wake up the frogs at $x$. This will trigger infinitely many frogs to visit
$x$ and therefore also $v$ so that $R_v$ occurs as well.

On $\Z^d$, $d\ge 2$, or in a more general setting the proof is not
that simple since it is not obvious that any woken up frog will a.s.\
hit a site $x$ for which $R_x$ occurs, if there are such sites. Since
the occurrence of the event $R_x$ may depend on trajectories of all
frogs, we need to introduce an ``extra frog'' and use
properties of the environment viewed from the extra frog to prove that
the extra frog hits a site $x$ for which $R_x$ occurs. After that, we
have to get rid of the extra frog.

In our second corollary the frogs are not independent under $P$.  We
consider random walks in a common environment of random conductances,
see \cite{Bis11} for a survey of this model.  Let $d\ge 2$ and assign
to each undirected edge $\{x,y\}$ between nearest neighbors
$x,y\in \Z^d$ a non-negative random variable $c(\{x,y\})$, called the
{\em conductance} of this edge.  Let
$q(x):=\sum_{y: \|y-x\|_1=1}c(\{x,y\})$ and set
$\varkappa(x,y):=c(\{x,y\})/q(x)$ for nearest neighbors $x$ and $y$ if
$q(x)\ne 0$, $\kappa(x,x)=1$ if $q(x)=0$, and $\varkappa(x,y):=0$ in
all other cases. Then $\varkappa$ is a random stochastic matrix. We
say that $x$ and $y$ are {\em connected} in environment $c$ if there is a
nearest neighbor path connecting $x$ and $y$ along which all
conductances are strictly positive. Let ${\cal C}(x)$ be the (random)
cluster consisting of $x$ and all points of $V$ which are connected to
$x$ in the environment $c$. Note that due to our definition of
recurrence only points in an infinite cluster can be recurrent.
\begin{theorem}[\rm\bf Random walks among random
  conductances] \label{rwrc} Assume that the family
  $(c(\{x,y\}))_{x,y\in\Z^d}$ is stationary and ergodic w.r.t.\ the
  shifts on $\Z^d$, $E[q(0)]<\infty$, and that there is a.s.\ at most
  one infinite cluster.\footnote{The last condition is satisfied, for
    instance, if $(\1_{c(\{x,y\})>0})_{x,y}$ is i.i.d..}  Given
  $(c(\{x,y\}))_{x,y}$ let the frog trajectories
  $(S_j(x,i))_{j\ge 0}$, $x\in \Z^d,\ i\ge 1$, be independent
  nearest-neighbor Markov chains with transition matrix $\varkappa$.
  Let the numbers $\eta(x)$, $x\in V,$ of sleeping frogs be i.i.d.\
  and independent of the conductances and the frog trajectories.  Then
  either with probability 1 every $x\in V$ is transient or with
  probability 1 every $x\in V$ with $\eta(x)\ge 1$ and
  $\#{\cal C}(x)=\infty$ is recurrent.
\end{theorem}
\begin{remark}
  {\em The recurrence question for a frog model is non-trivial only
    when the individual frogs in Theorem \ref{rwrc} are transient. For example, the
    simple random walk on the infinite Bernoulli percolation cluster
    is transient starting with $d\ge 3$ (\cite{GKZ93}). In the case
    when conductances are i.i.d., the individual frogs in 
    Theorem~\ref{rwrc} are transient as soon as the simple random walk
    on the infinite cluster is transient (\cite{PP96}).\hfill $\Box$}
\end{remark}
Among the first results on the frog model was \cite[Theorem 5]{TW99}, which states that the frog process on $\Z^d$ is recurrent  regardless of the dimension $d$ if there is initially one frog per site and the frogs perform independent simple symmetric random walks. Since this fits into the setting of Theorem \ref{rwrc} (with $c(\{x,y\})=1$ and $\eta(x)=1$) it raises the following question:
\begin{problem}{\rm Is the model in Theorem \ref{rwrc} always recurrent, i.e.\ is it always true that $P[R_0]=P[\eta(0)\ge 1,\#{\cal C}(0)=\infty]$?}
\end{problem}
Finally, we obtain a sufficient recurrence condition for a large class of models
on $\mathbb{Z}^d$ with independent frogs. The proof uses
our general zero-one law, Theorem \ref{012}.
\begin{theorem}\label{log}
  Let $d\ge 1$ and $\epsilon\in(0,1/(2d)]$ be fixed.  
Suppose that  the frog
  trajectories $(S_j(x,i))_{j\ge 0}$, $x\in \Z^d, i\ge 1$, satisfy 
\[P[S_{j+1}(x,i)=S_j(x,i)+e\, \mid\, (S_k(x,i))_{0\le k\le j}]\ge \eps\]
for all $j\ge 0$ and unit vectors $e\in\Z^d$.
 Let the numbers
  $\eta(x)$, $x\in V,$ of sleeping frogs be identically distributed
  and assume that the quantities $\eta(x), (S_j(x,i))_{j\ge 0}$,
  $x\in V, i\ge 1,$ are independent.
Then there is a constant $\con{so}$ which
  depends only on $d$ and $\epsilon$ such that if
\begin{equation}\label{dipl}
P[\eta(0)\ge t]\ge \frac{c_{\ref{so}}}{(\log t)^{d}}
\end{equation}
for all large $t$
 then with probability $1$ every $x\in\Z^d$ with $\eta(x)\ge 1$ is recurrent. For $d=1$ the tail condition {\em (\ref{dipl})} can be replaced by the weaker moment assumption $E\left[\log_+\eta(0)\right]=\infty.$
\end{theorem}
 In Proposition~\ref{sharp} below we show that in the
  case when  the frogs have a drift away from the origin the condition
\begin{equation}\label{ii}
E\left[(\log_+\eta(0))^d\right]=\infty
\end{equation}
is necessary for recurrence. Since \eqref{dipl} is only slightly stronger than \eqref{ii},
the condition \eqref{dipl} is close to being sharp, see also
Problem~\ref{ns}.
\begin{remark}
  {\em In the case when $S(x,i)$ is a nearest neighbor random walk
    with a non-zero drift in the first coordinate direction,
    recurrence results were previously obtained in \cite{GS09} ($d=1$)
    and \cite{DP14} ($d\ge 1$). Theorem~2.1 of \cite{DP14} states that
    if
  \begin{equation}
    \label{dp}
      E\left[(\log_+\eta(0))^{\frac{d+1}{2}}\right]=\infty
  \end{equation}
  then the frog model is recurrent. It was proved
  previously in \cite[Theorem 2.2]{GS09} that for $d=1$ the condition
  (\ref{dp}) is necessary and sufficient for recurrence. 
    Theorem~\ref{log} shows that the frog model can be recurrent even
    if the expectation in \eqref{dp} is finite, demonstrating that
    (\ref{dp}) is not necessary for recurrence when $d\ge 2$.
  \hfill\qed }
\end{remark}

\textbf{Organization of the paper.} In Section~\ref{1} we introduce
the general frog model and state our main result
Theorem~\ref{012}. The proof of this theorem is given in
Section~\ref{2}. Section~\ref{3} discusses a number of applications of
Theorem~\ref{012}, in particular, it contains the proofs of
Theorems~\ref{01} and \ref{rwrc}. Theorem~\ref{log} and its partial
converse are proven in Section~\ref{4}. Appendix \ref{A} discusses the
above mentioned 
conjecture from \cite{GS09}. Technical results are collected in
Appendices B and C.

\section{The frog model and the main result}\label{1}
Let $V$ be a countably infinite set.
A frog configuration $(n,s)$ on $V$ consists of 
\begin{itemize}
\item $n=(n(x))_{x\in V}$, where $n(x)\in\N_0:=\N\cup\{0\}$ represents the number
  of frogs at site $x$ indexed by $i=1,2,\dots,n(x)$ which can be
  woken up and will be called {\em initial frogs};
\item a collection of paths $s=(s_j(x,i))_{j\ge 0,x\in V,i\ge 1}$,
  where $s_j(x,i)\in V$ denotes the position of the $i$-th frog
  originating at $x$ after $j$ steps. We assume that $s_0(x,i)=x$.
\end{itemize}
The set of all frog configurations is denoted by $\mathbb{F}\subseteq \N_0^V\times V^{\N_0\times V\times\N}$.
Next we define recurrence of a site $v\in V$ for a given
configuration $(n,s)\in \mathbb{F}$. Set $W_0^v(n,s):=\{v\}$ and define recursively for $j\ge 1$,
\begin{align*}
  W_j^v(n,s):=\bigg\{x\in V\setminus \bigcup_{k=0}^{j-1}W_k^v(n,s)\,\bigg|\,&\exists k<j,\  y\in W_k^v(n,s),\ 1\le i\le n(y):\\
&
s_{j-k}(y,i)=x\bigg\}
\end{align*}
and let $W^v(n,s):=\bigcup_{j\ge 0}W_j^v(n,s).$
Note that if at time $0$ we wake up the  initial frogs at site $v$
then $W_j^v(n,s)$ is the set of sites visited at time $j$ for the
first time by an active frog and $W^v(n,s)$ is the set of sites ever
visited by an active frog provided that $n(v)\ge 1$.
\begin{defi}\label{rt}
  A site $v\in V$ is said to be \textbf{recurrent} for a frog configuration $(n,s)\in \mathbb{F}$ if
  there are infinitely many distinct $x\in W^v(n,s)$ for which there are $i\in\{1,\ldots,n(x)\}$ and $j\ge 0$ such that $s_j(x,i)=v$.
Otherwise $v$ is called \textbf{transient} for $(n,s)$.
\end{defi}
Note that if $n(v)=0$ 
then $v$ by the definition is transient for $(n,s)$.
\begin{remark}\label{rtdef}{\rm There are several possible ways to define
    recurrence of a site $v$ w.r.t.\ a frog configuration.  Our definition  appears to be the most restrictive, since we take into
    account only visits by active frogs originating at distinct
    sites. It might seem more natural to count the number of times
    when $v$ is occupied by any active frog.  Our choice is based on two reasons: (i) in essentially
    all cases of interest our definition is equivalent to less
    restrictive ones (see, for example, the discussion in
    Appendix~\ref{A}); (ii) it allows us to consider stopped frog
    trajectories (see (\ref{sf})) on the same state space without
    deactivating or removing the ``stopped frogs'' from the system.\hfill $\Box$}
\end{remark}
We consider random frog configurations. Equip $\mathbb F$ with its
standard $\si$-field. Let $(\Om,\F,P)$ be a probability space
and $(\eta,S)$ be an $\mathbb F$-valued random variable. We write
$\eta=(\eta(x))_{x\in V}$ and
$S=(S_j(x,i))_{j\ge 0,x\in V,i\ge 1}$, where $\eta(x)$ takes values
in $\N_0$ and $S_j(x,i)$ in $V$ with $S_0(x,i)=x$.  For our main result Theorem
\ref{012} we shall need several auxiliary random variables\footnote{In many applications some of these variables are trivial, i.e.\ $\cal{C}(v)=V$,
    $T(v,i)=\infty$, or $X_v=v$ a.s.\ for all $v\in V$, $i\ge 1$.
    But see Theorem~\ref{rwrc} and the proofs of Theorem~\ref{log} and
    Theorem~\ref{01q} where at least one set of these extra variables
    is non-trivially defined and plays an important role.} on
$(\Om,\F,P)$, namely,
\begin{itemize}
\item $V$-valued random variables $S_j(x,0)$, $j\ge 0, x\in V$,
  denoting the position at time $j$ of the so-called {\em extra frog}
  starting at site $S_0(x,0)=x$. We denote by
  $S^0:= (S_j(x,i))_{j\ge 0,x\in V,i\ge 0}$ the collection of frog
  trajectories including those of the extra frogs. The random variable
  $(\eta, S^0)$ takes values in 
  $\mathbb F^0\subseteq \N_0^V\times V^{\N_0\times V\times\N_0}$.
\item 
a family $\kappa=(\kappa(x,y))_{x,y\in V}$ of $[0,1]$-valued random variables,
  called {\em ellipticity variables}, which will provide bounds on
  transition probabilities (see assumption (EL) below).  We say that
  $y\in V$ {\em can be reached from} $x\in V$ if there exist an $m\ge 0$ and a
  sequence $(x=x_0,x_1,\ldots,x_{m}=y)\in V^{m+1}$ such that
  $\varkappa(x_{i-1},x_i)>0$ for all $i=1,\ldots,m$. We call
  $x,y\in V$ {\em equivalent w.r.t.}\ $\varkappa$ if $y$ can be reached
  from $x$ and $x$ from $y$.  The equivalence class of $v\in V$
  w.r.t.\ $\varkappa$ is denoted by $\cal{C}(v)$ and is called the {\em cluster}
  of $v$.
\item 
a family $T=(T(x,i))_{x\in V,i\ge 1}$ of $\N_0\cup\{\infty\}$-valued random variables. We think of $T(x,i)$ as the time at which the $i$-th
  frog originating at $x$ is stopped. However, we do not require $T(x,i)$ to be a stopping time.
\item a family $X=(X_v)_{v\in V}$ of $V$-valued random variables representing the
  choice of the extra frog which will be used to examine the
  recurrence/transience of the site $v$. 
\end{itemize} 
For $i,j\ge 0$ and $v\in V$ define the $\si$-field
  \begin{align*} {\cal F}_j(v,i):=\sigma\Bigg(&\kappa, T, \eta, S_m(x,k);\\ &(m,(x,k))\in \bigg(\N_0\times
    \left((V\times \N)\backslash\{(v,i)\}\right)\bigg) \cup
    \bigg(\{0,\ldots,j\}\times\{(v,i)\}\bigg)\Bigg).
\end{align*}
If $i\ge 1$ then ${\cal F}_j(v,i)$ contains all information about
$\varkappa$, $T$, and $\eta$, as well as the
information about the non-extra frogs with the exception of the steps
past $j$ of the $i$-th frog at $v$. Note that if $i=0$ then in
addition to all information about $\varkappa$, $T$, 
$\eta$, and all non-extra frogs  the $\sigma$-field
${\cal F}_j(v,i)$ contains the information about the first $j$ steps of
the extra frog at $v$.

 Next we state several conditions on the distribution of $(\kappa, T, X, \eta, S^0)$ under $P$.
Our first assumption is a mild ellipticity condition. It relates the  non-extra
frog trajectories and the ellipticity variables.
\[
\mbox{\begin{tabular}{l}
For all $v,y\in V,\ i\ge 1,$ and $j\ge 0$, $P$-a.s.\\
${\displaystyle \kappa(S_j(v,i),y)\le P\left[S_{j+1}(v,i)=y\,|\,{\cal F}_j(v,i)\right]\le \1_{\kappa(S_j(v,i),y)>0}.}$
\end{tabular}}
\tag{EL}\]

\begin{lemma} \label{ch}  For $v\in
  V$ let $R_v$ be the event that $v$ is recurrent for
  $(\eta,S)$.
 If {\rm (EL)} holds then $P[R_v\backslash\{\eta(v)\ge 1,\ \#\cal{C}(v)=\infty\}]=0$.
\end{lemma}
\begin{proof}
  As already noted after Definition \ref{rt},
  $R_v\subseteq \{\eta(v)\ge 1\}$. Moreover, by the upper bound in
  (EL), a.s.\ every point in $W^v(\eta,S)$ can be reached from $v$ and
  if $R_v$ occurs then, by Definition \ref{rt}, $v$ can be reached
  from infinitely many points of $W^v(\eta,S)$. Hence $R_v$ is a.s.\
  contained in $\{\#\mathcal C(v)=\infty\}$.
\end{proof}
\begin{defi}
  We say that the frog process satisfies the \textbf{zero-one law for
    recurrence and transience} if $P[R_v]=0$ for all $v\in V$ or
  $P[R_v]=P[\eta(v)\ge 1,\ \#\cal{C}(v)=\infty]$ for all $v\in V$.
\end{defi}
\begin{remark}\label{r6}{\rm By Lemma \ref{ch} the zero-one law holds
    iff either for all $v\in V$, $P[R_v]=0$ or for all $v\in V$ for
    which $P[\eta(v)\ge 1,\ \#\cal{C}(v)=\infty]>0$ we have
    $P[R_v\mid \eta(v)\ge 1,\ \#\cal{C}(v)=\infty]=1$. This justifies
    the name zero-one law. \hfill $\Box$}
\end{remark}
Our next assumption is the uniqueness of the infinite cluster. Together with (EL) it can be interpreted as an irreducibility assumption.
\[\tag{UC} \mbox{There is $P$-a.s.\ at most one infinite equivalence class w.r.t.\ $\varkappa$. }
\]
 Stopping the frog trajectories $S_\cdot(x,i)$ at the respective time $T(x,i)$ we obtain
\begin{equation}\label{sf}
  \bar S:=(\bar S_j(x,i))_{j\ge 0,x\in V,i\ge 1}:=(S_{j\wedge T(x,i)}(x,i))_{j\ge 0,x\in V,i\ge 1}.
\end{equation}
Denote by $\bar R_v$ the event that $v\in V$ is recurrent for $(\eta,\bar S)$.
Obviously, $\bar R_v\subseteq R_v$. The next assumption has to do with  a partial converse.
\[ \mbox{For all $v\in V$, $P[R_v]>0$ implies $P[\bar R_v]>0$.}
\tag{T}
\]
Note that (T) is void if $T\equiv\infty$.

To introduce our somewhat non-standard ergodicity assumption define
for each $\phi\in{\rm Sym}(V)$, i.e.\  permutation $\vp$ of $V$, the function
$\theta_\vp: \mathbb F\to\mathbb F$ by
\begin{equation}\label{perm}
  \theta_\vp(n,s):=\left(\left(n(\vp(x))\right)_{x\in V},\left(\vp^{-1}(s_j(\vp(x),i))\right)_{j\ge 0,x\in V,i\ge 1}\right)\in\mathbb F.
\end{equation}
Denote by $\cal I$ the $\si$-field of all $A\in\cal F$ for which there
is a measurable set $B$ such that 
$A\overset{P}{=}\left\{\theta_\vp(\eta, \bar S)\in B\right\}$ for all
$\vp\in{\rm Sym}(V)$, where $A\overset P= C$ means that $P[A\, \triangle\, C]=0$.
Such $A$ is called 
almost
invariant (w.r.t.\
${\rm Sym}(V)$).  Then the ergodicity assumption\footnote{Note that we
  do not require the distribution of $(\eta,\bar S)$ to be invariant w.r.t.\ $\theta_\phi$ for
  $\phi\in {\rm Sym}(V)$.} reads as follows: 
\[P[A]\in\{0,1\}\quad\mbox{for all $A\in\cal I$.}
\tag{ERG}\]
So far the distribution of the extra frogs has not played any role yet. 
The following assumption relates the extra frog and the first frog at $v\in V$.
\[\tag{EX}
\begin{tabular}{l}
\text{For all $v\in V$ the paths $(S_j(v,0))_{j\ge 0}$ and $(S_j(v,1))_{j\ge 0}$}\\
\text{are i.i.d.\
given $\cal F_0(v,1)$.}
\end{tabular}
\]
\begin{remark}\label{hai}{\rm Note that (EL), (UC), (T), and (ERG) are conditions on the distribution of $(\kappa, T, \eta, S).$
We claim that 
one can construct a probability space with random variables $\kappa, T, \eta$ and $S^0$ 
on which $(\kappa, T, \eta, S)$ has the same distribution as in the original model and, in addition, 
(EX) holds. In this sense (EX) can be assumed without loss of
generality. We shall prove this claim in Appendix \ref{ef}.\hfill
$\Box$ }
\end{remark}
The next assumption is obviously satisfied in the most common case
$X_v\equiv v$.
\[\tag{Xv}
\text{For all $v\in V$, $X_v$ and $(\kappa, T,\eta,S^0)$ are independent and $P[X_v=v]>0$.}
\]
Now we state the final  and crucial assumption.  Let
$o:V\to V$ be such that $o\circ o=o$. We call $o(x)$ the
representative of $x\in V$.\footnote{For a frog model on a transitive
  graph, $o(x)\equiv {\rm const}$ representing a reference point such
  as $0$ for $\Z^d$ or just an arbitrary fixed vertex. In general,
  when $V$ naturally splits into orbits under the action of some
  subgroup of ${\rm Sym}(V)$, $o(x)$ designates a reference point for
  each orbit. See Theorem~\ref{01q} and Example~\ref{co1}.}  For each
$x\in V$ fix some $\vp_x\in{\rm Sym}(V)$ such that $\vp_x(o(x))=x$.
We call
\begin{equation}\label{evp}
Z(x):=\left(\left(\1_{\bar R_{\vp_x(y)}}\right)_{y\in V},\left(\kappa\left(\vp_x(y),\vp_x(z)\right)\right)_{y,z\in V},o(x)\right)
\end{equation}
the {\em environment viewed from the vertex} $x\in V$.
We assume the existence of equivalent measures under which the environment viewed from the extra frogs always looks the same\footnote{Note that (ID) is weaker than the more common assumption
  that the environment viewed from the particle (here the extra frog)
  is stationary. See the proof of Theorem \ref{log} where we have an
  example where the latter assumption is not fulfilled but (ID) is.
  Counterexamples \ref{co2} and \ref{res} show that in Theorem
  \ref{012} the assumption (ID) cannot be replaced by the assumption
  that the environment is stationary w.r.t.\ to the canonical spatial
  shifts of the state space $V$.}.
\[\mbox{\begin{tabular}{l} For each $v\in V$ with $P[\#\cal{C}(X_v)=\infty]>0$ there is a\\ probability 
  measure $P_v$ which is  equivalent 
  to (i.e.\ mutually\\ absolutely continuous with) $P[\,\cdot\mid \#\cal{C}(X_v)=\infty]$ such that\\
      $Z\left(S_j(X_v,0)\right), j\ge 0,$ is identically distributed under $P_v$.
\end{tabular}}
\tag{ID}\]
Our main result is the following. 
\begin{theorem}[\rm \bf Zero-one law]\label{012}  
 Assume {\em  (EL), (UC), (T), (ERG), (EX), (Xv),} and {\em (ID)}.  
Then the zero-one law for recurrence and transience holds.
\end{theorem}
\section{Proof of Theorem \ref{012}}\label{2}
\begin{lemma}\label{nash}
Let $\vp\in{\rm Sym}(V)$, $v\in V$, and $(n,s)\in\mathbb F$. Then $\vp(v)$ is recurrent for $(n,s)$ iff $v$ is recurrent for $\theta_\vp(n,s)$.
\end{lemma}
\begin{proof}
For $x,y\in V$ and $(\tilde n,\tilde s)\in\mathbb F$ let $I_y^x(\tilde n,\tilde s):=1$ if there exist $j\ge 0$ and $1\le i\le \tilde n(y)$ such that $\tilde s_j(y,i)=x$. Otherwise, set $I_y^x(\tilde n,\tilde s):=0$.
Then by Definition \ref{rt},
\begin{equation}\label{zero}
\mbox{$x$ is recurrent for $(\tilde n,\tilde s)$\quad iff}
\sum_{y\in W^{x}(\tilde n,\tilde s)}I_y^{x}(\tilde n,\tilde s)=\infty.
\end{equation}
By induction over $j$ we see that 
$W_j^{\vp(v)}(n,s)=\vp\left(W_j^{v}\left(\theta_\vp(n,s)\right)\right)$  for all $j\ge 0$ and, consequently,
$W^{\vp(v)}(n,s)=\vp\left(W^{v}\left(\theta_\vp(n,s)\right)\right)$.
Moreover, 
$I_y^{\vp(v)}(n,s)=I_{\vp^{-1}(y)}^{v}(\theta_\vp(n,s)).$
Thus, by (\ref{zero}) the statement that $\vp(v)$ is recurrent for
$(n,s)$ is equivalent to
\[\sum_{\vp^{-1}(y)\in W^{v}(\theta_\vp(n,s))}I_{\vp^{-1}(y)}^{v}(\theta_\vp(n,s))=\infty.
\]
However, this is equivalent, again due to (\ref{zero}), to the recurrence of $v$ for $\theta_\vp(n,s)$.
\end{proof}
\begin{lemma}\label{ex2}
If {\em (EL)} and {\em (EX)} hold then for all $v,y\in V$ and $j\ge 0,$
\[P\left[S_{j+1}(v,0)=y\,|\,{\cal F}_j(v,0)\right]\ge
\kappa(S_j(v,0),y)\quad \text{$P$-a.s..}\]
\end{lemma}
We postpone the proof of Lemma \ref{ex2} to Appendix \ref{ef}.
\begin{proof}[Proof of Theorem \ref{012}] 
  We assume that there is a $u\in V$ such that $P[R_u]>0$ since
  otherwise there is nothing to prove.  We need to show that for all $v\in V$,
\begin{equation}\label{bb}
P[R_v]=P[\eta(v)\ge 1,\, \#\cal{C}(v)=\infty].
\end{equation}
Fix $v\in V$. Let 
\begin{equation}\label{go}
P[\#\cal{C}(v)=\infty]>0
\end{equation}
 since (\ref{bb}) is
trivial otherwise due to Lemma \ref{ch}.  We shall prove
  \eqref{bb} in five steps.

  {\em Step 1: For a.e.\ realization, the probability to reach a
    recurrent site from a fixed site $x$ of the infinite cluster is
    positive.} Define for all $x\in V$ the random variable
\[D(x):=\sup\left\{\frac {\1_{\bar R_{x_m}}}{m+1}\prod_{i=1}^{m}\kappa(x_{i-1},x_i)\ \Big|\  m\ge 0,  x=x_0, x_1,\ldots,x_m\in V\right\},\]
where $\sup\emptyset:=0$ and $\prod_{i=1}^0a_i:=1$. Note that  $D(x)=1$ iff the supremum in the definition above is attained for $m=0$, i.e.\ iff  $x$ is recurrent for $(\eta,\bar S)$. 
The  random variable $D(x)$ will serve as a lower bound on the probability
that an extra frog currently located at $x$ will ever reach a site $y$ which is recurrent for $(\eta,\bar S)$. We shall show that for all $x\in V$ such that $P[\#\cal{C}(x)=\infty]>0$,
\begin{equation}\label{do}
 P[D(x)>0\mid\#\cal{C}(x)=\infty]=1.
\end{equation}
Since $P[R_u]>0$ we have by assumption (T) that  $P[\bar R_u]>0$.
Therefore, also $P[\bar R]>0$, where $\bar R:=\bigcup_{w\in V}\bar R_w$. 
Due to Lemma \ref{nash}, $\bar R\in\cal I$. Consequently, by (ERG), 
$P[\bar R]=1.$ 
Hence there is a.s.\ a $w\in V$ for which $\bar R_w$ occurs and for which  by Lemma \ref{ch}, $\cal{C}(w)$ is infinite. If  $\cal{C}(x)$ is infinite then by (UC) a.s.\ $\cal{C}(x)=\cal{C}(w)$ and hence there is a path
connecting $x$ to  $w$ along which all
ellipticity variables $\varkappa$ are strictly
positive. Multiplying these variables yields (\ref{do}).

{\em Step 2: A modification of (\ref{do}) ``as seen from the extra frog'' holds.} More precisely, we shall show that 
\begin{equation}\label{MS}
P\left[\limsup_{j\to\infty}D(S_j(v,0))>0\ \Bigg|\ \#C(v)=\infty\right]=1.
\end{equation}
By (Xv), $X_v$ and $(\#\mathcal C(x), D(x))_{x\in V}$ are
independent and $P[X_v=v]>0$. Hence, due to (\ref{go}), $P[\#\cal{C}(X_v)=\infty]>0$
and,
 from (\ref{do}), $\vP[D(X_v)>0]=1$,
where $\vP:=P[\,\cdot\mid \#\cal{C}(X_v)=\infty]$. Therefore, since
$P_v\ll \vP$ by (ID),
\begin{equation}\label{geh}
P_v[D(X_v)>0]=1.
\end{equation}
Define  $f:\{0,1\}^V\times[0,1]^{V^2}\times V\to[0,1]$  by
\[f\left((r_x)_{x\in V},(a_{x,y})_{x,y\in V},y_0\right)
:=\sup\left\{\frac {r_{y_m}}{m+1}\prod_{i=1}^m a_{y_{i-1},y_i}\ \bigg|\ m\ge 0, y_1,\ldots,y_m\in V\right\}.\]
Then for all $x\in V$, by definition (\ref{evp}),
\[f(Z(x))=\sup\left\{\frac {\1_{\bar R_{\vp_x(y_m)}}}{m+1}\prod_{i=1}^m \kappa(\vp_x(y_{i-1}),\vp_x(y_i))\ \bigg|\ m\ge 0, o(x)=y_0,y_1,\ldots,y_m\in V\right\}.
\]
Replacing  $\vp_x(y_i)$ with $x_i$ and recalling that $\phi_x(o(x))=x$  we obtain 
$D(x)=f\left(Z(x)\right).$
Therefore, 
(ID) implies that $(D(S_j(X_v,0)))_{j\ge 0}$ is identically distributed under $P_v$.
Consequently, by continuity,
\begin{eqnarray*}
\lefteqn{P_v\left[\limsup_{j\to\infty} D(S_j(X_v,0))> 0\right]=
\lim_{K\to\infty}\lim_{J\to\infty}P_v\left[\sup_{j\ge J}D(S_j(X_v,0))> 1/K\right]}\\
&\ge&\lim_{K\to\infty}\limsup_{J\to\infty}P_v\left[D(S_J(X_v,0))> 1/K\right]\ =\
\lim_{K\to\infty}P_v\left[D(S_0(X_v,0))> 1/K\right]=1
\end{eqnarray*}
due to $S_0(X_v,0)=X_v$ and (\ref{geh}). 
 Therefore, $\vP$-a.s.\ $\limsup_{j\to\infty}D(S_j(X_v,0))>0$ because $\vP\ll
P_v$. Due to (Xv) 
and $P[X_v=v]>0$ this implies \eqref{MS}.

{\em Step 3: The extra frog which starts in the infinite cluster will
  hit a recurrent site.}  We shall argue that
\begin{equation}
  \label{ehit}
  P[\exists j\ge 0\ D(S_j(v,0)=1\,|\,\#{\cal C}(v)=\infty]=1.
\end{equation}
Define for $K\ge 1,$ $\tau_{0,K}:=-K$ and then recursively for $m\ge 1$,
\begin{eqnarray*}
\tau_{m,K}&:=&\inf\{j\ge\tau_{m-1,K}+K\,|\,D(S_j(v,0))>1/K\}.
\end{eqnarray*}
Note that $\tau_{m,K}$ is a stopping time w.r.t.\  the filtration $(\F_j(v,0))_{j\ge 0}$. 
Define the events
\[
A_{m,K}:=\bigcap_{j=0}^{K-1}\left\{ D\left(S_{\tau_{m,K}+j}(v,0)\right)<1, \tau_{m+1,K}<\infty\right\}\in\F_{\tau_{m+1,K}}(v,0).
\]
Roughly speaking, on the event $A_{m,K}$ the extra frog from site $v$ is at time $\tau_{m,K}$  not too far from a recurrent site, but still does not reach any recurrent site within the next $K-1$ steps.
Then for all $M\ge 0$ and $K\ge 2$,
\begin{equation}\label{PD}
P\left[\bigcap_{m=1}^{M} A_{m,K}\right]
=
E\left[P\left[A_{M,K}\mid \F_{\tau_{M,K}}(v,0)\right];\bigcap_{m=1}^{M-1} A_{m,K}\right].
\end{equation}
On the event $A_{M-1,K}$ we have $\tau_{M,K}<\infty$ and therefore
$D\left(S_{\tau_{M,K}}(v,0)\right)>1/K$. Consequently, there is a path
of length less than $K$ which starts at $S_{\tau_{M,K}}(v,0)$ and ends
at a recurrent site and along which the product of the ellipticity
variables is larger than $1/K$. By repeated application of Lemma \ref{ex2} we
conclude that the probability that the extra frog originating at $v$
does not reach a recurrent site in the next $K-1$ steps  after time $\tau_{M,K}$ is at most
$1-1/K$. Hence,
$P\left[A_{M,K}\mid \F_{\tau_{M,K}}(v,0)\right]\le 1-1/K.$ Induction
over $M$ then yields that the expression in (\ref{PD}) can be
estimated from above by $(1-1/K)^M$.  By letting $M\to\infty$ we
obtain for all $K\ge 1$,
\begin{equation}\label{PP}
P\left[\bigcap_{m\ge 1}A_{m,K}\right]=0.
\end{equation}
By (\ref{MS}) and continuity,
\begin{eqnarray*}\lefteqn{
  P\left[\forall j\ge 0\ D(S_j(v,0))<1\, \mid\,  \#C(v)=\infty\right]}\\
&=&\lim_{K\to\infty}P\left[\forall 
j\ge 0\ D(S_j(v,0))<1,\ \limsup_{k\to\infty}D(S_k(v,0))>1/K\ \Bigg|\ \#C(v)=\infty\right]\\
 &\le&\liminf_{K\to\infty}P\left[\bigcap_{m\ge 1}A_{m,K}\ \Bigg|\ \#C(v)=\infty\right] 
\overset{(\ref{PP})}{=}0.
\end{eqnarray*}
This implies \eqref{ehit}.

{\em Step 4: Adding the extra frog to a site in the infinite cluster makes that site recurrent.} By Step 3 we have  that if
$\mathcal C(v)$ is infinite then waking up the  initial frogs at $v$ and the extra frog at $v$ results $P$-a.s.\ in waking up the frogs in
at least one (random) site $z$ which is recurrent for   $(\eta,\bar S)$ and hence also recurrent for $(\eta,S)$. This in turn
causes frogs from infinitely many distinct sites to visit
$z$. Due to the lower bound in (EL) the same holds a.s.\ true
for any other point in $\cal C(z)$ as well.  The site $v$ is such a
point since $\cal C(z)$ is infinite by Lemma \ref{ch} and therefore identical to $\cal C(v)$ due to (UC).

 To phrase this conclusion more precisely, consider for $(n,s^0)\in\mathbb F^0$, $n=(n(x))_{x\in V}$, $s^0=(s_j(x,i))_{i,j\ge 0, x\in V}$, the configuration
\[
a_v(n,s^0):=\left(\left(n(x)+\1_{x=v}\right)_{x\in V},\left(s_j(x,i-\mathbf 1_{x=v})\right)_{j\ge 0, x\in V,i\ge 1}\right)\in\mathbb F,
\]
which we get by adding the extra frog at $v$ to the set of initial
frogs.  We have shown  above that 
\begin{equation}\label{80}
P\left[\mbox{$v$ is recurrent for $a_v(\eta,S^0)$}\mid \#\cal C(v)=\infty\right]=1.
\end{equation}
{\em Step 5: Removing the extra frog.} Consider for
$(n,s^0)\in\mathbb F^0$ also the configuration
\[
t_v(n,s^0):=\left(n,\left(s_j(x,i-\mathbf 1_{x=v, i=1})\right)_{j\ge 0, x\in V,i\ge 1}\right)\in\mathbb F,
\]
 which we get by replacing the first frog at $v$ by the extra frog at $v$. 
Observe that if $n(v)\ge 1$ then for all $j\ge 0$, 
\begin{equation}
  \label{incl}
  W_j\left(a_v(n,s^0)\right)\subseteq W_j\left(n,s\right)\cup
W_j\left(t_v(n,s^0)\right),
\end{equation}
where $s=(s_j(x,i))_{j\ge 0, x\in V,i\ge 1}$.
Indeed, each frog $\mathfrak{f}$ at a site from the set
$W_j\left(a_v(n,s^0)\right)$ was either activated at time $0$ or was
woken up by some frog, its ``ancestor'' (if $\mathfrak{f}$ was woken
up by several frogs simultaneously then we pick any one as its
``ancestor''). Following the ancestry line back to the time $0$ we can
identify which of the frogs woken up at time $0$ started this ancestry
line. If the line was started by the extra frog at $v$ then
$\mathfrak{f}$ is located at a site in $W_j\left(t_v(n,s^0)\right)$,
if not then we can say that $\mathfrak{f}$ is at a site from
$W_j\left(n,s\right)$.

Therefore, by (\ref{incl}), if $v$ is recurrent for $a_v(n,s^0)$ and
$n(v)\ge 1$, then it is recurrent for $(n,s)$ or for $t_v(n,s)$.
Hence, 
\begin{eqnarray}\nonumber
\lefteqn{P[\mbox{$\eta(v)=0$ or $\#\cal C(v)<\infty$}]}\\ \nonumber
&\overset{(\ref{80})}{\ge}&P[\mbox{$\eta(v)=0$ or $v$ is transient for $a_v(\eta,S^0)$}]\\ \nonumber
&\ge& P\left[\mbox{$v$ is transient for $\left(\eta,S\right)$ and for $t_v(\eta,S^0)$}\right]\\ 
\label{bo}
&=&E\left[P\left[R_v^c\cap\left\{\mbox{$v$ is transient for $t_v(\eta,S^0)$}\right\}\mid \mathcal F_0(v,1)\right]\right]\\  \nonumber
&\stackrel{\rm (EX)}{=}& E\left[P\left[R_v^c\mid \mathcal F_0(v,1)\right]P\left[\mbox{$v$ is transient for $t_v(\eta,S^0)$}\mid \mathcal F_0(v,1)\right]\right]\\ \nonumber
&\stackrel{\rm (EX)}{=}&E\left[P\left[R_v^c\mid \mathcal F_0(v,1)  \right]^2\right]\
\ge\ P[\mbox{$\eta(v)=0$ or $\#\cal C(v)<\infty$}],
\end{eqnarray}
where the last inequality holds since by Lemma \ref{ch} a.s.\ 
\begin{eqnarray}\label{mfo}
P\left[R_v^c\mid \mathcal F_0(v,1)  \right]&\ge& 
P\left[\mbox{$\eta(v)=0$ or $\#\cal C(v)<\infty$}\mid \mathcal F_0(v,1)  \right]\\
&=&\1_{\{\eta(v)=0\ \mbox{\scriptsize or}\ \#\cal C(v)<\infty\}}.\nonumber
\end{eqnarray}
Therefore, the inequalities in (\ref{bo}) are, in fact,
identities. Hence we also have $P$-a.s.\ equality in (\ref{mfo}).
Taking expectations yields (\ref{bb}).
\end{proof}

\section{Examples illustrating the scope of Theorem~\ref{012}}\label{3}
In this section we give examples of frog processes covered by our  main
result, Theorem \ref{012}, as well as some counterexamples which demonstrate that our
assumptions are essential for the validity of the zero-one law.  The
examples are split into two groups. The first group concerns models
where frog dynamics is a (quasi-) transitive Markov chain as well as an
example where our general setting goes beyond the chain quasi-transitivity
condition. The second group discusses models where the underlying frog
dynamics is random walk in random environment so that under the
averaged measure the processes are not markovian. For all examples in
the first group $\cal{C}(v)=V$ for all $v\in V$, while the second group
includes examples with degenerate environments such as random walks on
the infinite percolation cluster.

\subsection{Transitive Markov chains and more}
In this subsection we assume that
$K:V\times V\to[0,1]$ is a stochastic matrix. 
Let $\Phi$ be a subgroup of ${\rm Sym}(V)$ such that for all $\vp\in\Phi$,
\begin{equation}
K\left(\vp(u),\vp(v)\right)=K(u,v)\quad\mbox{for all $u,v\in V$.} \label{tra}
\end{equation}
For $x,y\in V$ we say that $x\sim y$ iff there is  a $\vp\in\Phi$ such that $\vp(x)=y$.
Since $\Phi$ is a group, $\sim$ is an equivalence relation. Let $[x]:=\{\vp(x)\,|\,\vp\in\Phi\}$ be the equivalence class of 
$x\in V$, also called the {\em orbit} of $x$. If $(X_n)_{n\ge 0}$ is a Markov chain with transition matrix $K$ then $([X_n])_{n\ge 0}$, the so-called {\em factor chain} (cf.\ \cite[(1.31)]{Woe00}),  
is a Markov chain with state space $\tilde V:=\{[x]\,|\, x\in V\}$ and  transition matrix 
\[\tilde K([x],[y]):=\sum_{z\in [y]}K(x,z).\]
\begin{theorem}\label{01q}
  Let $K$ be irreducible, all orbits $[x],\ x\in V,$ be infinite, and
  $\tilde K$ have an invariant probability measure $\tilde\mu$.
  Assume that the frog trajectories $(S_j(x,i))_{j\ge 0}$,
  $x\in V,\ i\ge 1$, are Markov chains with common transition matrix
  $K$.
  Suppose that
  $\eta(x)$ and $\eta(y)$ have the same distribution whenever
  $x\sim y$
and  that the quantities $\eta(x), (S_j(x,i))_{j\ge 0}$, $x\in V,\ i\ge 1,$ are independent.
Then the zero-one law for recurrence and transience holds.
\end{theorem}
\begin{proof}
  We check the assumptions of Theorem \ref{012}.  Assumption (EL)
  holds with $\varkappa(x,y):=K(x,y).$ Since $K$ is irreducible, (UC)
  is satisfied as well.  We set $T\equiv \infty$ so that (T) holds
  trivially.

For the proof of  (ERG)
we shall use Proposition \ref{iid} and
consider the independent random variables $H(x):=\left(\eta(x),\left(S_j(x,i)\right)_{j\ge 0, i\ge 1}\right)$, $x\in V$. 
For $\vp\in\Phi$, $m\ge 0$, and $(s_j(i))_{j\ge 0, i\ge 1}\in V^{\N_0\times \N}$ define
$g_\vp(m,(s_j(i))_{j\ge 0, i\ge 1}):=(m,(\vp^{-1}(s_j(i)))_{j\ge 0, i\ge 1})$. Since $\vp\in {\rm Sym(V)}$, the random variables  $H_\vp(x):=g_\vp(H(\vp(x))$, $x\in V,$ are independent as well.
Moreover, 
for all $\vp\in\Phi, x\in V$, and $i\ge 1$, 
$\left(\vp^{-1}(S_j(\vp(x),i))\right)_{j\ge 0}$ is a Markov chain
starting at $x$ with transition matrix $K$
 since for all $y\in V$,
\[P[\vp^{-1}(S_j(\vp(x),i))=y]=P[S_j(\vp(x),i))=\vp(y)]=K^j(\vp(x),\vp(y))=K^j(x,y).
\]
Therefore,  due to independence,
\begin{equation}\label{ny}
\left(H_\vp(x)\right)_{x\in V}=\theta_\vp(\eta,S)\overset{d}{=}(\eta,S)=(H(x))_{x\in V},
\end{equation}
and hence (\ref{dis}). By Proposition \ref{iid} we obtain (ERG). 

To satisfy assumption (EX) let the extra frogs $(S_j(v,0))_{j\ge 0}$, $v\in V,$ be Markov chains
with transition matrix $K$ which are independent of $(\eta, S)$. 

For the remaining conditions let $\mathcal O\subseteq V$ be a complete set of representatives of $\sim$ and $o:V\to \mathcal O$ be a map which assigns to $x\in V$ its representative $o(x)\in[x]$.  For all $x\in V$ choose $\vp_x\in\Phi$ so that $\vp_x(o(x))=x$. We let $X_v$, $v\in V,$ be independent of  $(\eta,S^0)$ so that  $P[X_v=r]=\tilde \mu([r])$ for all $r\in \{v\}\cup \mathcal O\backslash\{o(v)\}$. Note that the requirement $P[X_v=v]>0$ is fulfilled since $\tilde K$ is irreducible since $K$ is so. Therefore (Xv) holds.

Finally, we shall check (ID).  Fix $v\in V$ and set $P_v:=P$.
Then for all measurable $B\subseteq\{0,1\}^V, C\subseteq [0,1]^{V^2}$, $D\subseteq \mathcal O$, and all $j\ge 0$,
\begin{eqnarray*}\lefteqn{P\left[Z\left(S_j(X_v,0)\right)\in B\times C\times D\right]}\nonumber\\
&=& \sum_{x\in V}
P\left[\left(\1_{R_{\vp_x(y)}}\right)_{y\in V}\in B,\left(\kappa\left(\vp_x(y),\vp_x(z)\right)\right)_{y,z\in V}\in C, S_j(X_v,0)=x, o(x)\in D\right]\nonumber\\
&=& \sum_{x\in V}
P\left[\left(\1_{R_{\vp_x(y)}}\right)_{y\in V}\in B\right]\1_C\left(\left(\kappa(y,z)\right)_{y,z\in V}\right) 
P\left[S_j(X_v,0)=x, o(x)\in D\right]
\end{eqnarray*}
 by (\ref{tra}) and independence of $(\eta,S)$ from $(X_v,S_\cdot(\cdot,0))$. 
Due to Lemma \ref{nash} and (\ref{ny}), the above expression is equal to
\[
P\left[\left(\1_{R_{v}}\right)_{v\in V}\in B\right]\1_C\left(\left(\kappa(y,z)\right)_{y,z\in V}\right) 
P\left[o\left(S_j(X_v,0)\right)\in D\right].
\]
This does not depend on $j$ since $(o(S_j(X_v,0)))_{j\ge 0}$ is a stationary Markov chain.
Therefore, (ID) holds. Theorem \ref{012} now yields the claim.
\end{proof}
\begin{proof}[Proof of Theorem~\ref{01}] Since the frog trajectories are
  assumed to be transitive Markov chains, there is only one
  orbit. Thus Theorem \ref{01} follows from Theorem~\ref{01q}.
  \end{proof}
  \begin{ex}[\rm \bf Quasi-transitive Markov chains]\label{qtmc} 
      {\rm If in the setting of Theorem~\ref{01q} there are only finitely
      many orbits, i.e.\ the factor chain is finite, then there is
      always an invariant probability measure
      $\tilde{\mu}$.
      Thus the 0--1 law holds. 

      We remark that in this case the Markov chain $(V,K)$ is called
      {\em quasi-transitive} under the action of the group $\Phi$
      (see \cite[pp.\ 13,14]{Woe00}).  As a representative example,
      consider a periodic model on $\Z^d$. Namely, set $V=\Z^d $, fix
      a period $L\in\N,\ L\ge 2$, and let $\Phi$ be the group of all
      shifts by vectors from $L\Z^d$, so that $x\sim y$ iff
      $y-x\in L\Z^d$. Then any frog model with ``independent
      ingredients'' such that its frog number distributions and
      Markovian dynamics are compatible with the periodic structure
      (as required by Theorem~\ref{01q}) satisfies the 0--1 law for
      recurrence and transience. \hfill $\Box$}
\end{ex}
Our next example shows that the applicability of Theorem~\ref{01} goes
beyond the quasi-transitive setting.
\begin{ex}[\rm \bf Frogs on a comb I]\label{co1}
{\rm  Let $V:=\Z\times\N_0$ and $p_1,p_2>0$ such that $p_1+p_2<1$. Define for $x\in\Z$,
 \begin{eqnarray*}
 K((x,0),(x+1,0))&:=& p_1,\\
 K((x,0),(x-1,0))&:=& 1-p_1-p_2,\\
 K((x,y),(x,y+1))&:=& p_2\quad \mbox{for $y\ge 0$, and}\\
 K((x,y),(x,y-1))&:=& 1-p_2\quad \mbox{for $y\ge 1$.}
 \end{eqnarray*}
 Let the random quantities $\eta(v), (S_j(v,i))_{j\ge 0}, v\in V,\ i\ge 1,$ be independent and assume that
for each $y\ge 0$ the random variables $\eta(x,y),x\in\Z,$ are identically distributed.
Let $\Phi$ be the group of shifts ``along $\Z$'' of the form $(x,y)\mapsto (x+k,y), k\in\Z.$
They satisfy (\ref{tra}). The orbits are $\Z\times\{y\}$, $y\ge 0$. The transition matrix for the factor chain is given by
\begin{eqnarray*}
\tilde K(\Z\times\{y\},\Z\times\{y+1\})&=&p_2\qquad\mbox{and}\\
\tilde K(\Z\times\{y\},\Z\times\{(y-1)\vee 0\}\})&=&1-p_2.
\end{eqnarray*}
If $p_2<1/2$ then $\tilde K$ admits an invariant probability measure $\tilde \mu$ and Theorem \ref{01q} yields the 0--1 law for recurrence and transience.\hfill $\Box$}
\end{ex}
\begin{cex}[\rm\bf Frogs on a comb II]\label{co2}
{\rm 
If $p_2\ge 1/2$ in the setting of Example \ref{co1} then there is no invariant probability measure $\tilde \mu$ for $\tilde K$. We shall
show that for $p_2>1/2$ the 0--1 law can fail.  
Let $\eta(x,y)$, $x\in\Z,\ y\ge 1$, be identically distributed with
$E[\log_+\eta(x,y)]<\infty$ while $\eta(x,0)$, $x\in\Z$, be
super heavy-tailed so that
\begin{equation}\label{kel}
 \liminf_{k\to\infty}k\,P[\log\eta(x,0)\ge k]>-\log p_1.
\end{equation}
Then, on the one hand, for each $x\in\Z$ the probability, given
$\eta$, of the event that all frogs
$S((x,y),i), y\ge 0,\ 1\le i\le \eta(x,y)$, stay forever on the tooth
$\{x\}\times\N_0$ is bounded from below by
\begin{equation}\label{ihs}
  \prod_{y\ge 0}(1-a^{y+1})^{\eta(x,y)}\ge \exp\left(-\con{aa}\sum_{y\ge 0}\eta(x,y)a^y\right)
\end{equation}
for suitable constants $0<a<1$ and $c_{\ref{aa}}>0$. By Lemma~\ref{bau} ($d=1$) the right hand side of
(\ref{ihs}) is strictly positive. Note that on this event no site
of the tooth $\{x\}\times\N_0$ can be visited infinitely many times
(otherwise, by the ellipticity, the process would not stay on the
tooth). Since $x$ was arbitrary, $P[R_{v}]<P[\eta(v)\ge 1]$ for all $v\in V$.

On the other hand, $P[R_v]>0$ for all $v\in
V$. To show this, let $Z_n$, $n\ge 0$, be the number of active or just
woken up frogs at site $(n,0)$ at time $n$ if at time $0$ we wake up
all initial frogs at $(0,0)$.  As long as it is strictly positive, $(Z_n)_{n\ge 0}$ evolves like a
branching process with offspring distribution Bernoulli$(p_1)$ and
$\eta(n,0)$ immigrants at each time $n$. By \cite[Theorem 2.2]{Bau13} and (\ref{kel}) 
such Markov chain is transient.  Hence $P[\forall n\ge 0: Z_n>0]>0$.
Consequently, waking up the frogs at $(0,0)$ results with positive
probability in waking up the frogs at all sites $(n,0), n\ge 0$.  Any
frog originating at $(n,0), n\ge 0,$ visits $(0,0)$ after $n$ steps
with probability $(1-p_1-p_2)^n$. Moreover, since (\ref{kel}) implies
$E[\log_+\eta(x,y)]=\infty$ we obtain from Lemma~\ref{bau}
($d=1$) that $P$-a.s.\ $\sum_{n}(1-p_1-p_2)^n\eta(n,0)$
diverges. Therefore, by the second part of the Borel-Cantelli lemma
$P[\cdot\mid\eta]$-a.s.\ (and hence also $P$-a.s.) infinitely many
frogs originating at $(n,0), n\ge 0,$ would visit $(0,0)$ in the
shortest possible time, if woken up. We conclude that $P[R_{(0,0)}]>0$
and therefore, by ellipticity, $P[R_{v}]>0$ for all $v\in V$.\hfill $\Box$
}
\end{cex}
\begin{problem}{\rm 
  Study the recurrence and transience of the above frog model on the
  comb for $p_2\le 1/2$.}
\end{problem}
  \begin{cex}[\rm\bf Not identically distributed frog numbers on
    $\Z^d$]\label{pop} {\rm \cite[Theorem 1.3 (ii)]{Pop01} provides
      examples where the frogs perform independent simple symmetric
      random walks on $V=\Z^d$, $d\ge 3$, and the 0--1 law fails. In
      these examples, all the assumptions of Theorem \ref{01} are
      satisfied except that $(\eta(x))_{x\in \Z^d}$ is not identically
      distributed. }
\end{cex}
\subsection{Random walk in random environment (RWRE)}
In this subsection we consider frogs which jump as  independent random walks in a common  random environment
on $V=\Z^d$ for some $d\ge 1$.
Fix $d\ge 1$ and let
\[
\Pi:=\Bigg\{(\pi(x,y))_{x,y\in\Z^d}\in [0,1]^{\Z^d\times\Z^d}\mid
\forall x\in\Z^d: \sum_{y\in\Z^d}\pi(x,y)=1\Bigg\}
\]
be the set of all stochastic matrices on $\Z^d$.
We endow $\Pi$
with the standard Borel $\si$-field $\mathcal B(\Pi)$.
An  element $\pi\in\Pi$ is called  a {\em random walk environment}. 
A time homogeneous Markov chain on $\Z^d$ with transition matrix $\pi$ is called a {\em random walk in the environment} $\pi$.
For the following we need some more notation. 

For all $z\in\Z^d$ we define by $\vp_z(x):=x+z$ the shift $\vp_z$ by
$z$ on $\Z^d$. This shift can also be applied to families of the form
$f=(f(x))_{x\in\Z^d}$ and $g=(g(x,y))_{x,y\in\Z^d}$ by
$\vp_z(f):=(f(x+z))_{x\in\Z^d}$ and
$\vp_z(g):=(g(x+z,y+z))_{x,y\in\Z^d}$, respectively. It can also be
applied to finite sequences $F=(f_1,\ldots,f_n)$ of such families by
setting $\vp_z(F):=(\vp_z(f_1),\ldots,\vp_z(f_n))$.  Moreover, such
$F$ is called {\em ergodic w.r.t.\ the shifts on $\Z^d$} if
$P[A]\in\{0,1\}$ for all events $A$ for which there is a measurable
set $B$ such that
$A \overset P=\left\{\vp_z(F)\in B\right\}$ for all
$z\in\Z^d$.

The {\em random walk environment viewed from the particle} is a Markov chain with state space $\Pi$ and transition kernel $K$ defined by
\begin{equation}\label{KK}
K(\pi,B):=\sum_{z\in\Z^d}\pi(0,z)\1_B\left(\vp_z(\pi)\right),
\qquad \pi\in\Pi, B\in\mathcal B(\Pi).
\end{equation}
Note that if  $(S_j)_{j\ge 0}$ is a random walk in the environment
  $\pi$ starting at $0$, then $(\phi_{S_j}(\pi))_{j\ge 0}$ is a Markov
chain on $\Pi$ with kernel $K$ starting from $\pi$.

For our general result below we
need to augment the random walk environment with the numbers of
initial frogs at each site. Therefore, we consider
$\bar \Pi:=\N_0^{\Z^d}\times \Pi$, endow it with its standard Borel
$\si$-field $\mathcal B(\bar\Pi)$ and call its elements
$\bar
\pi=(n,\pi)=\left((n(x))_{x\in\Z^d},(\pi(x,y))_{x,y\in\Z^d}\right)$
{\em augmented environments}.  The {\em augmented environment viewed
  from the particle} is a Markov chain with state space $\bar \Pi$ and
transition kernel $\bar K$ defined by
\[
\bar K(\bar \pi,\bar{B}):=\sum_{z\in\Z^d}\pi(0,z)\1_{\bar{B}}\left(\vp_z(\bar\pi)\right),\qquad
\bar \pi=(n,\pi)\in\bar \Pi, \bar{B}\in\mathcal B(\bar \Pi).
\]
As above, if $\bar \pi=(n,\pi)\in\bar \Pi$ and $(S_j)_{j\ge 0}$ is a random walk in the environment
$\pi$ starting at $0$, then 
\begin{equation}\label{mc}
(\phi_{S_j}(\bar\pi))_{j\ge 0}\quad\text{is a Markov
chain on $\bar\Pi$   with kernel $\bar K$ starting from $\bar\pi$.} 
\end{equation}
\begin{theorem}[\rm\bf RWRE]\label{rm01}
Suppose that the family $\varkappa=(\kappa(x,y))_{x,y\in\Z^d}$ of ellipticity variables takes values in 
$\Pi$ and satisfies {\rm (UC)} with $P[\#\cal{C}(0)=\infty]>0.$
Assume that $\bar \kappa:=\left((\eta(x))_{x\in\Z^d}, \kappa\right)$ is stationary and ergodic w.r.t.\ the shifts on $\Z^d$.  Suppose that 
 given $\bar \kappa$,  the sequences $(S_j(x,i))_{j\ge 0}, x\in \Z^d,\ i\ge 1$, are independent Markov chains with common transition matrix $\varkappa$.
Finally, assume that there is  a probability measure  $\bar P$ on $\bar \Pi$ which is invariant w.r.t.\ $\bar K$ and equivalent to the distribution of $\bar\kappa$  under $P[\,\cdot\mid\#\cal{C}(0)=\infty]$. 
Then the zero-one law for recurrence and transience holds.
\end{theorem}
\begin{proof} We check the assumptions of Theorem \ref{012}.  (EL) is
  fulfilled by construction. We let $T\equiv\infty$ and  $X_v:=v$ for all $v\in \Z^d$ so that (T) and (Xv)
  trivially hold.  Denote by $\mathcal M(\Z^d)$ the set of all
  probability measures on $\Z^d$.  Choose
  $f:\mathcal M(\Z^d)\times[0,1]\to\Z^d$ such that for all
  $\mu\in\mathcal M(\Z^d)$ and all random variables $X$ which are
  uniformly distributed on $[0,1]$, $f(\mu,X)$ has distribution $\mu$.
  Then we may assume without loss of generality that for all
  $i,j\ge 0, x\in\Z^d$,
\begin{equation}\label{sea}
S_{j+1}(x,i)=S_{j}(x,i)+f\left(\left(\kappa\left(S_j(x,i),S_j(x,i)+y\right)\right)_{y\in\Z^d},U_j(x,i)\right),
\end{equation}
where $U_j(x,i),\ i,j\ge 0,\ x\in\Z^d$, are independent and uniformly
distributed on $[0,1]$ and independent of $\bar \kappa$. Note that
(EX) holds.

We are left to check (ERG) and (ID). For $x\in\Z^d$ let
$o(x):=0\in\Z^d$ and
$U(x):=\left(U_j(x,i)\right)_{j\ge 0, i\ge 1}$. Set
$U:=(U(x))_{x\in\Z^d}$.
Since
$U$ is i.i.d.\ and independent of $\bar\kappa$, $(U,\bar\kappa)$ is stationary
and ergodic\footnote{The proof is similar to the one of
  $(4)\Rightarrow (5)$ in \cite[Theorem 6.1, p.\,65]{Pet83}.} w.r.t.\
the shifts on $\Z^d$. Moreover, there is a deterministic measurable
function $g_1$ such that for all $z\in\Z^d$,
\begin{equation} \label{quak}
\theta_{\vp_z}(\eta,S)\overset{(\ref{perm})}
=\left(\vp_z(\eta),(S_j(x+z,i)-z)_{j\ge 0, x\in\Z^d,i\ge 1}\right)
= g_1\left(\vp_z(U,\bar\kappa)\right).
\end{equation}
Indeed, fix $i\ge 1$ and $x,z\in\Z^d$ and abbreviate
$\Delta_j:=S_j(x+z,i)-z$ for all $j\ge 0$. Then due to (\ref{sea}) 
for all $j\ge 0$,
\[\Delta_{j+1}=\Delta_j+f\left((\kappa\left(\Delta_j+z,\Delta_j+y+z\right))_{y\in\Z^d},U_j(x+z,i)\right).
\]
It follows by induction over $j$ that $\Delta_j$ is a function of
$j,x,i$ and $\vp_z(U,\bar\kappa)$. This implies (\ref{quak}).
Thus, (ERG) follows from the ergodicity of $(U,\bar\kappa)$.

For the proof of (ID) it suffices due to stationarity to consider the
case $v=0$.  We construct the measure $P_0$ on $(\Omega,{\cal F})$ as follows. Set $\wP:=P[\,\cdot\mid \#\cal{C}(0)=\infty]$.  By the
Radon-Nikodym theorem there is a strictly positive density
$h:=d \bar P/d\wP_{\bar\varkappa}:\bar \Pi\to(0,\infty)$, where
$\wP_{\bar\varkappa}$ denotes the distribution of $\bar\varkappa$
under $\wP$.  Define $P_0[A]:=\wE\left[h(\bar\kappa); A\right]$ for all
$A\in\F$ and note that  
\begin{equation}\label{note}
\mbox{the distribution of $\bar\kappa$ under $P_0$ is $\bar{P}$.}
\end{equation}
 As required, $P_0$ and $\wP$ are equivalent. 

Observe that the environment viewed from $z\in \Z^d$
as defined in (\ref{evp}) can be written as
$Z(z)=\left(\left(\1_{R_{x+z}}\right)_{x\in
    \Z^d},\vp_z(\kappa),0\right).$
Therefore, there is due to Lemma \ref{nash} and (\ref{quak}) a
deterministic function $g_2$ such that
$Z(z)=g_2\left(\vp_z(U,\bar\kappa)\right)$
for all $z\in\Z^d$.
Abbreviate $S_j:=S_j(0,0)$. We need to show that the distribution of $Z(S_j)=g_2(\vp_{S_j}(U,\bar\kappa))=g_2(\vp_{S_j}(U),\vp_{S_j}(\bar\kappa))$
under $P_0$
does not depend on $j$. 
Note that $U, U(0,0),$ and $\bar\kappa$ are independent under $P$ and hence also under $P_0$. Moreover, $U$ is stationary under $P$ and so it is under $P_0$. 
Since $S_j$ is measurable w.r.t.\ $\si(U(0,0),\bar\kappa)$, this implies  that
$Z(S_j)$ has the same distribution under $P_0$ as $g_2(U,\vp_{S_j}(\bar\kappa))$. Therefore and again by independence it suffices to show that $(\vp_{S_j}(\bar\kappa))_{j\ge 0}$ is stationary under $P_0$. This is the case due to (\ref{mc}), (\ref{note}) and the invariance of $\bar P$ w.r.t.\ $\bar K$.   
Consequently, (ID) is fulfilled and Theorem
\ref{012}  yields the claim.
\end{proof}
\begin{proof}[Proof of Theorem \ref{rwrc}] The theorem trivially holds
  if $P[\#{\cal C}(0)=\infty]=0$. Therefore, we assume that
  $P[\#{\cal C}(0)=\infty]>0$. Since $(\eta(x))_{x\in\Z^d}$ is i.i.d.\
  and independent of the conductances and of the frog trajectories,
  $\bar\varkappa$ is stationary and
  ergodic\footnote{See the previous footnote.}.  We only need to
  produce an invariant measure $\bar
  P$ which is equivalent to $\wP_{\bar
    \varkappa}$, the distribution of $\bar
  \kappa$ under $\wP:=P[\,\cdot\mid\mathcal \#\cal
  C(0)=\infty]$. First, define
  \[Q[B]:=\frac{E[q(0);\{\kappa\in B\}\cap\{\#{\cal
      C}(0)=\infty\}]}{E[q(0);\#{\cal C}(0)=\infty]},\quad B\in{\cal
    B}(\Pi).\]
  Then $Q$ is equivalent to $\wP_\varkappa$ since $q(0)>0$ on the event $\{\#{\cal C}(0)=\infty\}$.
  Moreover, $Q$ is invariant for $K$ defined in (\ref{KK}) due to the following variation of a standard argument, cf.\ \cite[Lemma 2.1]{Bis11}.  
For every bounded measurable function
  $f:\Pi\to \R$ 
 we have by  stationarity of the  conductances
\begin{align*}
  E&[q(0);\#{\cal C}(0)=\infty]\int_{\Pi}(Kf)(\pi)\,dQ(\pi) 
=\sum_{\|e\|_1=1} E\left[c(\{0,e\}) f(\phi_e(\varkappa));\#{\cal C}(0)=\infty\right]\\ &=
\sum_{\|e\|_1=1} E\left[c(\{-e,0\}) f(\varkappa);\#{\cal C}(-e)=\infty\right]=
\sum_{\|e\|_1=1} E\left[c(\{e,0\}) f(\varkappa);\#{\cal C}(e)=\infty\right]\\ &=
\sum_{\|e\|_1=1} E\left[c(\{0,e\}) f(\varkappa);\{\#{\cal C}(e)=\infty\}
\cap\{c(\{0,e\})>0\}\right]\\ &
            =\sum_{\|e\|_1=1} E\left[c(\{0,e\}) f(\varkappa);\#{\cal C}(0)=\infty\right]
=E[q(0);\#{\cal C}(0)=\infty]\int_\Pi f(\pi)\,dQ(\pi).
\end{align*}
  Denote by $P_\eta$ the distribution of $\eta$ under $P$.
Then $\bar P:=P_\eta\otimes Q$ is invariant w.r.t.\
  $\bar K$ and equivalent to $P_{\bar \varkappa}$. The claim now
  follows from Theorem~\ref{rm01}.
\end{proof}

\begin{ex}[\rm\bf RW in i.i.d.\ environment]\label{r1} {\rm There are
    several examples of RWRE known in which the random environment
    $(\kappa(x,x+\cdot))_{x\in\Z^d}$ is i.i.d.\ and for which there is
    a measure $Q$ on $\Pi$ that is invariant w.r.t.\ $K$ defined in
    (\ref{KK}) and absolutely continuous w.r.t.\ $P_\varkappa$, see
    e.g.\ \cite{BS02a} and the references mentioned therein after
    (0.12) and in the introduction of \cite{Sab13}. For more recent
    results see, e.g.\ \cite[Theorem 1(i)]{Sab13} and \cite{BCR14}.
    In these cases the environment is elliptic, i.e.\
    $\varkappa(x,y)>0$ a.s.\ for all nearest neighbors $x,y\in\Z^d$
    such that $\cal{C}(0)=\Z^d$.  It has been noticed, e.g.\ in
    \cite[Lemma 4]{Sab13}, see also \cite[Lecture 1, Theorem
    1.2]{BS02b}, that such $Q$ is automatically equivalent to
    $P_\varkappa$.   Also in the case of balanced and
      possibly non-elliptic environments considered in \cite{BD14}
      there exists such an equivalent $Q$ provided that
      there is at least one direction $e_i\in\Z^d$ such that
      $\varkappa(x,x+e_i)>0$ a.s.\ for all $x$ \cite{Ber16}.

 In all these cases, if we choose $(\eta(x))_{x\in\Z^d}$ i.i.d.\ and
    independently of $\kappa$, then, as in the proof of Theorem
    \ref{rwrc}, the measure $Q$ can be augmented to a probability
    measure which is invariant w.r.t.\ $\bar K$ so that Theorem
    \ref{rm01} is applicable.\hfill $\Box$ }
\end{ex}

\begin{cex}[\rm\bf RWRE dynamics, (ID) is not satisfied]
\label{res} 
{\rm This example shows that in Theorem \ref{rm01} the requirement of the existence of an invariant probability measure $\bar P$ cannot be dropped.
In \cite{BZZ06}, a $\Pi$-valued environment $\kappa$ on $\Z^d, d\ge 3$, is constructed which is stationary and ergodic w.r.t.\ the shifts on $\Z^d$ and also uniformly elliptic\footnote{i.e.\ there is an $\eps>0$ such that $\varkappa(x,y)>\eps$ $P$-.a.s.\ for all nearest neighbors $x,y\in\Z^d$}, but for which the corresponding RWRE $(X_j)_{j\ge 0}$ starting at $0$ disobeys the so-called 0--1 law for directional transience. (For $d=2$ such an example is constructed in \cite{Hei13}. For a simpler, but not uniformly elliptic example see \cite[Section 3]{ZM01}.) In particular, there are 
$\{0,1,2\}$-valued random variables $N(x)$, $x\in\Z^d$, such that $(\varkappa,N)$ is stationary and ergodic w.r.t.\ the shifts and 
\begin{equation}\label{12}
P[N(X_j)=i\ \mbox{for all $j\ge 0$}]>0
\end{equation}
for $i=1,2$
(see the proof  of \cite[Theorem 3]{BZZ06} on page 847). Now let $(\tilde \eta(x))_{x\in\Z^d}$ be i.i.d.\ and independent of $(\kappa,N)$ such that $E[\log_+(\tilde\eta(0))]=\infty$. Set $\eta(x):=\tilde \eta(x)\1_{N(x)=1}$. 
The 0--1 law for recurrence and transience fails in this case.

Indeed, on the one hand if $\eta(0)\ge 1$ and  if the first frog woken up at 0 stays forever in the set $\{x\,|\, N(x)=1\}$ then it will wake up for each $m\ge 1$ at least an independent $\tilde\eta(0)$-distributed number of frogs at $\ell_1$-distance $m$ from 0. Each such frog  has a chance of at least $\eps^m$ to reach the origin. Lemma \ref{bau} ($d=1$) and the Borel-Cantelli lemma then imply  $P[R_0]\ge P[\eta(0)\ge 1, \forall j\ge 0: N(S_j(0,1))=1]$, which is strictly positive due to (\ref{12}) for $i=1$.  

On the other hand, $P[R_0]< P[\eta(0)\ge 1]$. Indeed, there is a.s.\ a
finite nearest-neighbor path from 0 to a site $v$ with $N(v)=2$. If
$N(0)<2$ and if we wake up the frogs at 0 then with positive
probability they and all the frogs woken up by them will follow this
path to $v$ and then stay forever within $\{x: N(x)=2\}$ without ever
returning to 0 due to (\ref{12}) for $i=2$.  \hfill $\Box$ }
\end{cex}

\section{Recurrence and transience of some inhomogeneous frog models on
  $\Z^d$}\label{4}
\begin{proof}[Proof of Theorem \ref{log}]
  We generate the trajectories $(S_j(x,i))_{j\ge 0}$,
  $x\in\Z^d,\ i\ge 0,$ in the following way.  Enumerate the $2d$ unit
  vectors of $\Z^d$ as $e_1,\ldots, e_{2d}$.  Let $U_j(x,i),$
  $j\ge 0, x\in\Z^d,\ i\ge 0,$ be i.i.d.\ random
  variables, each one uniformly distributed on $[0,1]$.   Since for each $x\in\Z^d$ and $i\ge 1$,
  $((S_k(x,i))_{0\le k\le j})_{j\ge 0}$ is a Markov chain (with state
  space $\bigcup_{j\ge 0}(\Z^d)^{j+1}$) we may assume without loss of
  generality that there are functions
  $f_{j,x,i}:\Z^{d(j+1)}\times[0,1]\to\Z^d, j\ge 0, x\in\Z^d,\ i\ge
  1,$ such that 
\[S_{j+1}(x,i)=f_{j,x,i}\left((S_k(x,i))_{0\le k\le j},U_j(x,i)\right)\]
and $f_{j,x,i}((x_k)_{0\le k\le j},u)=x_j+e_m$ if $u\in[(m-1)\varepsilon, m\varepsilon)$ for some $m\in\{1,\ldots,2d\}$.
We define the extra frogs by setting  for $x\in\Z^d,\ j\ge 0,$
\[
S_{j+1}(x,0):=f_{j,x,1}\left((S_k(x,0))_{0\le k\le j},U_j(x,0)\right).
\]
Let
$T(x,i):=\inf\{j\ge 0: U_j(x,i)\ge 2d\varepsilon\},\ 
    x\in\Z^d,\ i\ge 1$. Then the random variables $T(x,i)$, $x\in\Z^d,\ i\ge 1$
    are independent and geometrically distributed with parameter
    $1-2d\varepsilon$.
Finally, define  for all $x,y,v\in \Z^d$, $o(x):=0\in\Z^d$, $X_v:=v$, and set $\varkappa(x,y):=\varepsilon$ if $\|x-y\|_1=1$ and  $\varkappa(x,y):=0$ otherwise.

We claim that this collection of random variables satisfies the
assumptions of Theorem \ref{012}.  It is obvious that the assumptions
(EL), (UC), (EX) and (Xv) hold. As in the proof of Theorem \ref{01q}, Assumption (ERG) is satisfied since the walks
$(\bar S_j(x,i))_{j\ge 0},\ x\in\Z^d,\ i\ge 1,$ are independent, simple symmetric
random walks on $\Z^d$ which are stopped after the i.i.d.\ times $T(x,i)$. For the same reason $(Z(x))_{x\in\Z^d}$ is stationary. Moreover, $(Z(x))_{x\in\Z^d}$ is independent of the extra frogs. Hence (ID) holds
(even though the environment viewed from the extra frog is not
stationary in general).    Assumption (T) is satisfied since
$P[\bar R_v]=P[\bar R_0]$ for all $v\in V$ and
\begin{equation}\label{bel}
P[\bar R_0]>0
\end{equation}
as we shall show below.

Therefore, Theorem \ref{012} applies. Since $P[R_v]\ge P[\bar R_v]>0$ due to (\ref{bel}), we obtain $P[R_v]= P[\eta(v)\ge 1]$, i.e.\ the claim of Theorem \ref{log}.

It remains to show (\ref{bel}).
To this end we shall first show that if the frogs at 0 are woken up and all the woken up frogs are stopped after their respective times $T(x,i)$, it happens with positive probability that  every site is eventually visited, i.e.
\begin{equation}\label{kn}
P\left[W^0(\eta,\bar S)=\Z^d\right]>0.
\end{equation}
For $x\in \Z^d$ and $r\ge 0$ let
$B(x,r):=\left\{y\in \Z^d: \|y-x\|_1\le r\right\}$ and
$\partial B(x,r):=\left\{y\in \Z^d: \|y-x\|_1= r\right\}$
be the ball and sphere, respectively, with center $x$ and $\ell_1$-radius $r$. Then
\begin{eqnarray*}
P\left[W^0(\eta,\bar S)=\Z^d\right]
&=& P\left[\forall r\ge 1:\ B(0,r)\subseteq W^0(\eta,\bar S)\right]
\ =\ \prod_{r\ge 1}(1-a_r),\quad\mbox{where}\\
a_r&:=&P\left[\partial B(0,r)\not\subseteq W^0(\eta,\bar S)\mid B(0,r-1)\subseteq 
W^0(\eta,\bar S)
\right].
\end{eqnarray*}
Note that $a_r<1$.
Therefore, the above product does not vanish iff the sequence $(a_r)_{r\ge 1}$ is summable. We estimate
\begin{eqnarray*}\nonumber
a_r&\le&
\sum_{y\in \partial B(0,r)}P\left[y\notin W^0(\eta,\bar S)\mid B(0,r-1)\subseteq W^0(\eta,\bar S)
\right]\\
&\le &\sum_{y\in \partial B(0,r)}P\left[y\notin \left\{\bar S_j(z,i): j\ge 0, z\in B(y,r)\cap B(0,r-1), 1\le i\le\eta(z)\right\}\right]\\
&= &\sum_{y\in \partial B(0,r)} \prod_{z\in B(y,r)\cap B(0,r-1)}E\left[P\left[
y\notin \left\{\bar S_j(z,i): j\ge 0, 1\le i\le\eta(z)\right\}\mid \eta(z)\right]\right]\\
&\le &\sum_{y\in \partial B(0,r)}\prod_{k=1}^r \prod_{z\in \partial B(y,k)\cap B(0,r-1)}
E\left[(1-\varepsilon^k)^{\eta(z)}\right]\\
&\le& c_{\ref{dd}}r^{d-1}\prod_{k=1}^r 
\left(E\left[(1-\varepsilon^k)^{\eta(0)}\right]\right)^{c_{\ref{cc}}k^{d-1}}
\ \le\ c_{\ref{dd}}r^{d-1}\prod_{k=1}^r 
\left(\vp(\varepsilon^k)\right)^{c_{\ref{cc}}k^{d-1}},
\end{eqnarray*}
where $c_{\ref{dd}}(d)<\infty$ and  $c_{\ref{cc}}(d)>0$ are constants such that $\#\partial B(0,r)\le \con{dd}r^{d-1}$ and
$\#\partial B(y,k)\cap B(0,r-1)\ge \con{cc}k^{d-1}$ for all $1\le k\le r$ and  $y\in \partial B(0,r)$ and where
$\vp(t):=E[e^{-t\eta(0)}]$ denotes the Laplace transform of $\eta(0)$ at $t\ge 0$. 
It is well-known that $P[\eta(0)\ge t^{-1}]\le 2(1-\vp(t))$ for all $t>0$, see, e.g.\ \cite[Lemma 5.1 (3)]{Kal02}.
Therefore, it suffices to show that
\[ \sum_{r\ge 1}b_r<\infty\qquad\mbox{where}\qquad b_r:=c_{\ref{dd}}r^{d-1}\prod_{k=1}^r \left(1-P[\eta(0)>\varepsilon^{-k}]/2\right)^{c_{\ref{cc}}k^{d-1}}.
\]
We are going to apply Raabe's criterion. Due to (\ref{dipl}) we have for large $r$,
\begin{eqnarray*}
r\left(\frac{b_{r+1}}{b_r}-1\right)
& \le& r\left(
\left(1+\frac 1r\right)^{d-1}\left(1-\frac{c_{\ref{so}}}{2(r+1)^d(\log \varepsilon^{-1})^d}\right)^{c_{\ref{cc}}(r+1)^{d-1}}-1\right)\\
&=&r\left(\left(1+\frac{d-1}{r}+o\left(r^{-1}\right)\right)\left(1-\frac{c_{\ref{so}}\con{neu}(d,\varepsilon)}{r+1}+o\left(r^{-1}\right)\right)-1\right)\\
&=&r\left(\frac{d-1}{r}-\frac{c_{\ref{so}}c_{\ref{neu}}}{r+1}+o\left(r^{-1}\right)\right)
\begin{array}[t]{c}{\longrightarrow}\vspace*{-2mm}\\{\scriptstyle r\to\infty}\end{array}
 d-1-c_{\ref{so}}c_{\ref{neu}},
\end{eqnarray*}
which is less than $-1$ for $c_{\ref{so}}$ sufficiently large.  Consequently, $\sum_rb_r$ and $\sum_ra_r$ are finite and  (\ref{kn}) follows.

 Having proven \eqref{kn}, we obtain (\ref{bel}) if we show that waking up all initial frogs on $\Z^d$ causes
a.s.\ infinitely many of them to visit 0 before they are
stopped. Define the independent events
$A(x,i):=\left\{\exists j\ge 0: \bar S_j(x,i)=0\right\},$
$x\in\Z^d,\ i\ge 1,$ and observe that
$P[A(x,i)\mid \eta]=P[A(x,i)]\ge \varepsilon^{\|x\|_1}$
a.s.\ for all $x,i$. Moreover, (\ref{dipl}) implies 
  \eqref{ii}.  Therefore, a.s.
\[\sum_{x\in\Z^d, 1\le i\le \eta(x)}P[A(x,i)\mid\eta]\ge \sum_{r\ge 0}\varepsilon^r\sum_{\|x\|_1=r}\eta(x)=\infty\]
due to Lemma \ref{bau}.
By the second part of the Borel-Cantelli lemma $P[\,\cdot\mid \eta]$-a.s.\  infinitely many of the events $A(x,i), x\in\Z^d, 1\le i\le \eta(x),$ occur. Consequently,  $P$-a.s.\ infinitely many of the events $A(x):=\cup_{i\le \eta(x)}A(x,i)$ occur.  This proves (\ref{bel}) and concludes the proof of the theorem  for $d\ge 2$.

For the final claim regarding $d=1$ note that in this case every frog trajectory starting at 0 is a.s.\ infinite and therefore covers a.s.\ $\N$ or $-\N$.
Consequently, waking up a frog at 0 results a.s.\ in waking up all frogs on $\N$ or $-\N$.
As above, 
this implies due to 
$E[\log_+\eta(0)]=\infty$ that infinitely many frogs visit 0.
\end{proof}
\begin{problem}\label{ns}{\rm Does Theorem \ref{log} still hold true 
   if we replace the tail condition (\ref{dipl}) by the weaker moment
    condition (\ref{ii})?}
\end{problem}
If the answer to  Problem \ref{ns}  were positive then condition (\ref{ii}) in Theorem \ref{log} would be  sharp as the following result shows.  
\begin{prop}\label{sharp}
  Let $d\ge 1$, $(\eta(x))_{x\in\Z^d}$ be identically distributed, and
\begin{equation}\label{di}
E\left[(\log_+\eta(0))^d\right]<\infty.
\end{equation}
Suppose that $K$ is a stochastic matrix on $\Z^d$ for which there exist $M\in \N$ and $\delta\in(0,M]$ such that $K(x,y)=0$ if $\|x-y\|_1\ge M$ and
\begin{equation}\label{dri}\bigg\|\sum_{y\in\Z^d}K(x,y)y\bigg\|_1\ge \|x\|_1+\delta \qquad\mbox{for all $x\in \Z^d$}.
\end{equation}
  Assume that the frogs
$S_\cdot(x,i)$, $x\in\Z^d,\ i\ge 1,$ evolve as Markov chains with
transition matrix $K$ starting at $x$. Moreover, let the quantities
$\eta(x)$, $S_\cdot(x,i),\ x\in\Z^d,\ i\ge 1,$ be independent.  Then
waking up all the initial frogs on $\Z^d$ results a.s.\ only in
finitely many visits to 0. In particular, 0 is a.s.\ transient.
\end{prop}
\begin{proof}
  Let $A(x,i):=\left\{\exists j\ge 0: S_j(x,i)=0\right\}$ for
  $x\in\Z^d,\ i\ge 1.$ Then for every $a>1$
\begin{equation}\label{eth}
P[A(x,i)]\le \sum_{j\ge 0}P[S_j(x,i)=0]=\sum_{j\ge 0} P[a^{-\|S_j(x,i)\|_1}\ge 1]\le \sum_{j\ge 0}E[a^{-\|S_j(x,i)\|_1}].
\end{equation}
For $a> 1$ let $\ell_a:\R\to\R$ be the linear function  whose graph
passes through $(-M,a^M)$ and $(M,a^{-M})$. By convexity,
$a^{-s}\le \ell_a(s)$ for all $s\in[-M,M]$. Note that $\ell_a(s)<1$
for all $s>x_a:=M(a^M-1)(a^M+1)^{-1}$ and that $x_a\searrow 0$ as $a\searrow 1$. Therefore, given $\delta$ as in \eqref{dri},  there is an $a>1$ such that $\ell_a(\delta)<1$. Choose such an $a$. Then 
for any Markov chain $(S_j)_{j\ge 0}$ with transition matrix $K$ and
$S_0=x$,
\begin{eqnarray*}
  E[a^{-\|S_j\|_1}]&=&E\left[E[a^{-(\|S_j\|_1-\|S_{j-1}\|_1)} \mid S_{j-1}]
a^{-\|S_{j-1}\|_1}\right]\\&\le& 
E\left[E[\ell_a(\|S_j\|_1-\|S_{j-1}\|_1) \mid S_{j-1}]a^{-\|S_{j-1}\|_1} \right]\\
                   &\overset{(\ref{dri})}{\le}&\ell_a(\delta)E[a^{-\|S_{j-1}\|_1}]\ \le\ (\ell_a(\delta))^ja^{-\|x\|_1}
\end{eqnarray*}
by induction over $j$.  Consequently, there is due to (\ref{eth}) a
constant $\con{xx}$ such that a.s.\
$P[A(x,i)]\le c_{\ref{xx}}a^{-\|x\|_1}$ for all
$x,i$.  Therefore, by independence, a.s.
\[\sum_{x\in\Z^d, 1\le i\le \eta(x)}P[A(x,i)\mid\eta]=\sum_{x\in\Z^d, 1\le i\le \eta(x)}P[A(x,i)]\le c_{\ref{xx}}\sum_{r\ge 0}a^{-r}\sum_{\|x\|_1=r}\eta(x)<\infty\]
 due to Lemma \ref{bau} and (\ref{di}).  By the Borel-Cantelli lemma
$P[\,\cdot\mid \eta]$-a.s.\ only finitely many of the events
$A(x,i), x\in\Z^d, 1\le i\le \eta(x),$ occur. This implies the claim.
\end{proof}  Note that Theorem \ref{log} and Proposition
  \ref{sharp} resemble Theorems 1.12 and 1.10 in \cite{AMP02}
  respectively. In \cite{AMP02} frogs have finite,
  geometric life times.
\appendix
\section{Discussion of  Conjecture 2 from \cite{GS09}}\label{A}
Conjecture 2 in \cite{GS09} states that a frog model on a vertex
transitive graph $G=(V,E)$ such that the frog numbers $\eta(x)$,
$x\in V\setminus\{o\}$, are i.i.d., there is exactly one active frog at $o$
at time $0$, and the frog dynamics are given by
independent homogeneous random walks satisfies the following zero-one
law for recurrence and transience: either with probability one $o$ is
visited infinitely often or with probability one $o$ is visited only
finitely many times.

If all $\eta(x)$, $x\in V$, were i.i.d.\ and  the
  notions of recurrence and transience were defined as 
  in Definition~\ref{rt} then Theorem~\ref{01} would simply cover
 this conjecture, since a homogeneous walk on a vertex
transitive graph is a transitive Markov chain.   However,
  also the original conjecture follows from Theorem~\ref{01} as we
  shall show now.

 Assume without loss of generality that
  $P[\eta(x)\ge 1]>0\ (x\ne o)$ and that the Markov chain for the individual
  frog dynamics in the \cite{GS09} model is transient.  Use the same
  ingredients in Theorem~\ref{01}. Denote by $R_v$ the event that $v$
  is recurrent according to Definition~\ref{rt}. Suppose that the
  conclusion of Theorem~\ref{01} is that $P[R_v]=0$ for all $v\in V$.
  Then $P[R_o\,|\,\eta(o)\ge 1]=0$.  By monotonicity in the
  initial number of frogs, the fact that the number of frogs at
  each site is finite, and the assumed transience of the individual
  frog dynamics, the probability that $o$ is visited infinitely many
  times (by any frog) starting with a single frog at $o$ is equal
  to $0$ as well.

Suppose now that the conclusion of Theorem~\ref{01} is that
$P[R_v]=P[\eta(v)\ge 1]$ for all $v\in V$. It follows that for every
$m\in\N$ such that $P[\eta(o)=m]>0$ we have
$P[R_0\,|\,\eta(o)=m]=1$. Let
$M:=\min\{m\in\N:\,P[\eta(o)=m]>0\}$. If $M=1$ then the proof of the
conjecture is complete. Suppose $M>1$. We claim
that if we follow, say, the first one of these $M$ frogs and let the
rest $M-1$ frogs remain inactive forever then $o$ will still be
recurrent with probability one.  The argument follows closely the
procedure of removing the extra frog in the proof of Theorem~\ref{012}
after (\ref{80}). Set $P_M[\,\cdot\,]:=P[\,\cdot\,|\,\eta(o)=M]$ and
denote by $(n,s)^{(i)}$, $i\in\{1,\dots,M\}$, the frog configuration
which is obtained from $(n,s)$ by making all but the $i$-th frog at
site $o$ inactive. Note that as in (\ref{incl}) if $o$ is recurrent
for $(n,s)$ then $o$ is recurrent for at least one of $(n,s)^{(i)}$,
$i\in\{1,\dots,M\}$.  Denote by ${\cal G}$ the $\sigma$-algebra
generated by $(\eta(x),S(x,i)),\ x\in V\setminus\{o\},\ i\ge 1$. Then
by the above observation and independence
  \begin{align*}
    0&=P_M[R_o^c]\ge P_M\left[\cap_{i=1}^M\{o\text{ is transient for }(\eta,S)^{(i)}\}\right]\\&=E_M\left[P_M\left[\cap_{i=1}^M\{o\text{ is transient for }(\eta,S)^{(i)}\}\big|\,{\cal G}\right]\right]\\&=E_M\left[\left(P_M\left[o\text{ is transient for }(\eta,S)^{(1)}\big|\,{\cal G}\right]\right)^M\right].
  \end{align*}
  This gives that
  $P_M\left[o\text{ is transient for }(\eta,S)^{(1)}\big|\,{\cal
      G}\right]=0$
  $P_M$-a.s., and after integration we obtain
  that $P_M[o\text{ is transient for }(\eta,S)^{(1)}]=0$. The
  proof of the conjecture is now complete.

\section{Technical results}\label{B}
The following result generalizes \cite[Proposition 7.3]{LP}.
\begin{prop}[\rm\bf i.i.d.\ $\Rightarrow$ ergodic]\label{iid}
  Let $V$ be a countably infinite set and $\Phi\subseteq {\rm Sym}(V)$ be closed under composition and such that
  all orbits $\Phi_x:=\{\vp(x)\,|\, \vp\in\Phi\},\ x\in V,$ are
  infinite.  Let $H(x),\ x\in V,$ be independent random variables
  with values in some measurable space $(\mathbb S,\cal S)$ and
  ${\cal F}=\sigma(H(x),\ x\in V)$.  For each $\vp\in\Phi$ let
  $g_\vp:\mathbb S\to\mathbb S$ be a measurable map and define
  $H_\vp(x):=g_\vp(H(\vp(x)))$, $x\in V$. Finally suppose that
\begin{equation}\label{dis} 
  \forall \vp_1,\vp_2\in\Phi \ \ (H_{\vp_1}(x))_{x\in V}\overset d=(H_{\vp_2}(x))_{x\in V}.
\end{equation}
Then for all $A\in\cal
I:=\left\{A\in\F\,|\,\exists B\in\cal S^{\otimes V}\ \forall
    \vp\in\Phi:\ A
\overset P=
\left\{\left(H_\vp(x)\right)_{x\in V}\in
      B\right\}\right\}$ we have that $P[A]\in\{0,1\}$.
\end{prop}
For the proof we shall need the following fact.
\begin{lemma}\label{dj}
  Let $V$ and $\Phi$ satisfy the assumptions of
  Proposition~\ref{iid}.
  Then for each pair of finite subsets $F$ and $J$ of $V$ there
  is a $\vp\in\Phi$ such that $\vp[F]\cap J=\emptyset$.
\end{lemma}
\begin{proof}
  We use induction over $k$, the number of elements in $F$. The case
  $k=0$ is trivial. Assume as induction hypothesis that the statement
  has been shown for some $k\ge 0$ and let $F\subset V$ have exactly
  $k+1$ elements.  The proof is by contradiction. Suppose that
  $J\subset V$ satisfies
\begin{equation}\label{elm}
\vp[F]\cap J\not=\emptyset\quad\mbox{ for all $\vp\in\Phi$.}
\end{equation}
Let $(v_n)_{n\ge 1}$ enumerate $V$, set $J_n:=J\cup\{v_1,\ldots,v_n\}$
for all $n\ge 1$, and choose $v\in F$. Then $F':=F\backslash\{v\}$ has
exactly $k$ elements. By the induction hypothesis for each
$n\ge 1$ there is a $\vp_n\in\Phi$ such that
\begin{equation}\label{st}
\vp_n[F']\cap J_n=\emptyset
\end{equation} 
and, therefore, also $\vp_n[F']\cap J=\emptyset$. Then $\vp_n(v)\in J$
for all $n\ge 1$ due to (\ref{elm}). Since $J$ is finite, by the
pigeon hole principle there is a $u\in J$ such that $\vp_n(v)=u$ for
infinitely many $n\ge 1$.  Since $\Phi_u$ is infinite, there is a
$\psi\in\Phi$ such that $\psi(u)\notin J$. Choose $n$ large enough so
that $\psi^{-1}[J]\subseteq J_n$ and $\vp_n(v)=u$. Set
$\vp:=\psi\circ\vp_n\in\Phi$. Then
$\vp(v)=\psi(\vp_n(v))=\psi(u)\notin J$ and, by (\ref{st}),
$\vp[F']=\psi[\vp_n[F']]\subseteq\psi[J_n^c]=\psi[J_n]^c\subseteq
J^c$.
Thus $\vp[F]\subseteq J^c$, and we get a contradiction with (\ref{elm}).
\end{proof}
\begin{proof}[Proof of Proposition \ref{iid}]
  Let $(v_n)_{n\ge 1}$ be an arbitrary enumeration of $V$. Fix
  $A\in \mathcal I$ and let $B$ be a measurable set such that
  $A
\overset P=
\left\{\left(H_\vp(x)\right)_{x\in V}\in B\right\}$ for all
  $\vp\in\Phi$.  Set
  $\F_n(\vp):=\si\left(H_\vp(v_k); 1\le k\le n\right)$ and
  $Z_n(\vp):=P[A\mid\F_n(\vp)]$ for all $\vp\in\Phi$ and $n\ge 1$.
  Fix $\eps>0$ and $\bar{\vp}\in \Phi$.  By Levy's 0--1 law,
  $Z_n(\bar{\vp})$ converges in $\mathcal L^1$ to $\1_A$ as
  $n\to\infty$.  Hence there is an $n\ge 1$ such that
\begin{equation}\label{m1}
E\left[|Z_n(\bar{\vp})-\1_A|\right]\le\eps/2.
\end{equation}
Set $V':=\{v_k: k\le n\}$.
By Lemma~\ref{dj} there is a $\vp'\in\Phi$ such that
 \begin{equation}\label{pd}
  V'\cap \vp'[V']=\emptyset.
 \end{equation}
 We claim that the distribution of $(\1_A,Z_n(\vp))$ is the same for
 all $\vp\in\Phi$.  Indeed,
a.s.
 \begin{align*}
   (\1_A,Z_n(\vp))=&(\1_B((H_\vp(x))_{x\in V}),E[\1_B((H_\vp(x))_{x\in
   V})\,|\,H_\vp(v_k),\,1\le k \le n])\\=&f_n((H_\vp(x))_{x\in V})
 \end{align*}
 for some measurable function $f_n:\mathbb{S}^V\to \{0,1\}\times[0,1]$
 which does not depend on $\vp$ due to (\ref{dis}). The claim follows
 by another application of (\ref{dis}).  Therefore, by (\ref{m1})
\begin{equation}\label{m2}
  E\left[|Z_n(\bar{\vp}\circ\vp')-\1_A|\right]\le\eps/2.
\end{equation}
Consequently,
\begin{eqnarray*}
  \lefteqn{\left|P[A]-P[A]^2\right|=\left|E\left[\left(\1_A-Z_n(\bar{\vp})+Z_n(\bar{\vp})\right)\1_A\right]-P[A]^2\right|}\\
&\le&E\left[\left|\1_A-Z_n(\bar{\vp})\right|\1_A\right]+\left|E\left[Z_n(\bar{\vp})\left(\left(\1_A-Z_n(\bar{\vp}\circ\vp')\right)+Z_n(\bar{\vp}\circ\vp')\right)\right]-P[A]^2\right|\\
&\le&E\left[\left|\1_A-Z_n(\bar{\vp})\right|\right]+E\left[|Z_n(\bar{\vp})|\left|\1_A-Z_n(\bar{\vp}\circ\vp')\right|\right]+\left|E\left[Z_n(\bar{\vp})Z_n(\bar{\vp}\circ\vp')\right]-P[A]^2\right|.
\end{eqnarray*}
Due to (\ref{m1}), the first summand in the line above is less than or equal to $\eps/2$.
The same holds for the second term due to $|Z_n(\bar{\vp})|\le 1$ a.s.\ and (\ref{m2}).
And the last summand vanishes since $Z_n(\bar{\vp})$ and $Z_n(\bar{\vp}\circ\vp')$ are independent because $\cal F_n(\bar{\vp})$ and $\F_n(\bar{\vp}\circ\vp')$ are independent due to (\ref{pd}).
Letting $\eps\searrow 0$ proves the claim.
\end{proof}
\begin{lemma}\label{bau}
Let $d\in\N, 0<c<\infty,$ and $0<a<1$. Assume that $(Y_{i,n})_{i,n\ge 0}$ is an i.i.d.\ family of non-negative random variables. Then 
a.s.,
\begin{equation}\label{as}
E\left[\left(\log_+Y_{0,0}\right)^d\right]<\infty\qquad \mbox{iff}\qquad
\sum_{n\ge 0}a^n\sum_{i=0}^{\lfloor c n^{d-1}\rfloor}Y_{i,n}<\infty.
\end{equation}  
\end{lemma}
For $d=1$, this lemma
 is well-known, see \cite[Theorem 5.4.1]{Luk75}
or \cite[Exercise 22.10]{Bil95}.
For $d\ge 1$ it follows from \cite[Theorem 2]{Zer02}, see \cite[Lemma 2.2]{Bau14} for details 
and note that by Kolmogorov's 0--1 law the double sum in (\ref{as}) converges with probability 0 or 1. 
\section{Adding extra frogs}\label{ef}
\begin{lemma}\label{wlog1}
  Let $(\Omega^0,{\cal F^0},P^0)$ be a probability space, $(\mathbb{X},{\cal X})$ a measurable space and $\mathbb S$  a Polish space with Borel $\si$-field $\cal S$. Suppose $X^0:\Om^0\to\mathbb X$ and $Y^0:\Om^0\to\mathbb S$ are
  measurable.
  Then there is a probability space $(\Om,\F,P)$ and
  measurable maps $X,Y$, and $Z$ on $(\Omega,{\cal F})$ such that
\begin{itemize}
\item[(a)]
$(X^0,Y^0)\overset d=(X,Y)$ and
\item[(b)]  $Y$ and $Z$ are i.i.d.\ given $X$.
\end{itemize}
\end{lemma}
\begin{proof} 
 There is a probability kernel $\mu:\mathbb X\times\mathcal S\to[0,1]$ such that for all $B\in\mathcal S$, $P^0$-a.s.\ $P^0[Y^0\in B\mid X^0]=\mu(X^0,B)$ (see, e.g.\ \cite[Theorem 6.3]{Kal02}).
Let $\Om:=\mathbb X \times \mathbb S\times \mathbb S$, $\F$ be the product $\si$-field on $\Om$, denote by $E^0$ the expectation operator w.r.t.\ $P^0$, and define for 
  $A\in {\cal X}$ and $B,C\in{\cal S}$,
\[
P[A\times B\times C]:=E^0\left[\mu(X^0,B)\, \mu(X^0,C);\, X^0\in A\right].
\]
By \cite[Theorem 11.3]{Bil95}, $P$ can be extended to a probability measure on $\F$. We let $(X,Y,Z)$ be the identity on $\Om$. Then (a) follows from our choice of $\mu$. Moreover, it follows from $X^0\stackrel{d}{=}X$ that $P$-a.s., $P[Y\in B, Z\in C|X]=\mu(X,B)\mu(X,C).$
This implies (b).
\end{proof}
\begin{lemma}[\rm \bf Coupling]\label{kell} 
Let $V$ be countably infinite and suppose that  $\mathbb U$ and $\mathbb S$  are Polish spaces with Borel $\si$-fields $\cal U$ and $\cal S$, respectively. Assume  that for each $v\in V$  there is a probability space $(\Om_v,\F_v,P_v)$,  and random variables $U_v:\Om_v\to\mathbb U$ and $Z_v:\Om_v\to\mathbb S$ such that  $U_v$, $v\in V$, are identically distributed. 

Then there is a probability space $(\Om,\F,P)$ and random variables $U:\Om\to\mathbb U$ and $E_v:\Om\to\mathbb S$, $v\in V$, on it such that 
\begin{equation}\label{gg}
(U,E_v)\overset d=(U_v,Z_v)\quad\mbox{ for all $v\in V$.}
\end{equation}
\end{lemma}
\begin{proof} 
For all $I\subseteq V$ define the product measurable space
$(\Om_I,\F_I):=(\mathbb U,\cal U)\otimes(\mathbb S, \cal S)^{\otimes I}$. Choose $(\Om,\F):=(\Om_V,\F_V)$ and let $(U,(E_v)_{v\in V})$ be the identity on $\Om$.
Denote by $\nu$ the common distribution of $U_v$, $v\in V$. 
 For each $v\in V$ there is a probability kernel  $\mu_v:\mathbb{U}\times{\cal S}\to[0,1]$ such
  that for all
  $B\in {\cal S}$, $P_v$-a.s.\ $\mu_v(U_v,B)=P_v[Z_v\in B\,|\,U_v]$.
For finite $I\subseteq V$ and $A\in\cal U, B_v\in\cal S$ we set
\[P_I\left[A\times\prod_{v\in I}B_v\right]:=\int_A \prod_{v\in I}\mu_v(x,B_v)\ d\nu(x).\]
This defines a projective family of probability measures $P_I$
on $(\Om_I,\F_I)$. We choose $P$ as its projective limit, see, e.g.\
\cite[Theorem 6.14]{Kal02}. Then (\ref{gg}) follows from the
construction.
\end{proof}

\begin{cor}\label{wlog}
Let $V$ be countably infinite and let $\mathbb W$ and $\mathbb S$  be Polish spaces with Borel $\si$-fields $\cal W$ and $\cal S$, respectively.
   Suppose we are given a probability space  $(\Omega^0,{\cal F^0},P^0)$ and random variables $W^0:\Om^0\to\mathbb W$ and $F^0_v:\Om^0\to\mathbb S$ for each $v\in V$. 

  Then there is a probability space $(\Om,\F,P)$ and random variables 
   $W:\Om\to\mathbb W$ and $F_v, E_v:\Om\to\mathbb S$, $v\in V,$ such that
\begin{itemize}
\item[(a)] $\left(W^0,(F^0_x)_{x\in V}\right)\overset d=\left(W,(F_x)_{x\in V}\right)$ and
\item[(b)]  for all $v\in V$, $F_v$ and $E_v$ are i.i.d.\ given $(W,\check F_v)$, where  $\check F_v:=(F_x)_{x\in V\backslash\{v\}}$.
\end{itemize}
\end{cor}
\begin{proof}
Fix $v\in V$. Endow $\mathbb X:=\mathbb W\times \mathbb S^{V\backslash\{v\}}$ with its product $\si$-field $\cal X$. Then $X^0:=(W^0,\check F^0_v):\Om^0\to\mathbb X$ is measurable. Set $Y^0:=F^0_v$ and apply Lemma \ref{wlog1} to these quantities. Since we fixed $v$, the outcome may depend on $v$, which we indicate by a subscript $v$. Hence we get from Lemma \ref{wlog1} for every $v\in V$ a probability space $(\Om_v,\F_v,P_v)$ and
  measurable maps $X_v,Y_v$, and $Z_v$ on $(\Omega_v,{\cal F_v})$ such that 
\begin{itemize}
\item[(a\textprime)]
$\left(W^0,(F_x^0)_{x\in V}\right)=(X^0,F^0_v)\overset d=(X_v,Y_v)$ and 
\item[(b\textprime)]  $Y_v$ and $Z_v$ are i.i.d.\  given $X_v$.
\end{itemize}
Now we set $\mathbb U:=\mathbb W\times \mathbb S^V$, denote by $\cal U$ the product $\si$-field on $\mathbb U$ and interpret $U_v:=(X_v,Y_v)$ as random variable on $\Om_v$ with values in $\mathbb U$. By (a\textprime) the $U_v, v\in V,$ are identically distributed. Therefore, Lemma \ref{kell} can be applied and yields random variables $U=:\left(W,(F_x)_{x\in V}\right)$ and $E_v, v\in V,$ on some probability space $(\Om,\F,P)$
such that by (\ref{gg}),
$\left((W,\check F_v),F_v,E_v\right)\overset{d}{=}(X_v,Y_v,Z_v).$
Therefore, (a) follows from (a\textprime) and (b) from (b\textprime).
\end{proof}
\begin{proof}[Proof of Remark \ref{hai}] We apply
    Corollary \ref{wlog} to
    $W^0:=\left(\varkappa, T, \eta,(S_\cdot(\cdot,i))_{i\ge 2}\right)$
    and the first frogs $F_v^0:=S_\cdot(v,1), v\in V,$ and obtain a
    new probability space and random variables
    $(W,(F_x)_{x\in V},(E_x)_{x\in V})$ on it, which satisfy (a) and
    (b). The trajectory of the extra frog at $v\in V$ is defined by
    $S_\cdot(v,0):=E_v$.  Statement (b) then means (EX).  
\end{proof} 
\begin{proof}[Proof of Lemma \ref{ex2}] Fix $v\in V$, set
$W_v:=\left(\varkappa, T,\eta,(S_\cdot(x,i))_{x\in V,\,i\ge 1+\1_{x=v}}\right)$,
$F_v:=S_\cdot(v,1),$
$E_v:=S_\cdot(v,0),$  
and let $f$ be a measurable function on $V^{\N_0}$.  First, we show that for all measurable sets $C\subseteq V^{\N_0}$,
\begin{equation}\label{srw}
P[E_v\in C\mid W_v,F_v,f(E_v)]=P[E_v\in C\mid W_v,f(E_v)]\quad\text{$P$-a.s..}
\end{equation}
For \eqref{srw} it suffices to check that for sets $G$ of the form $\{W_v\in A, F_v\in B, f(E_v)\in D\}$, where  $A, B,$ and $D$ are measurable sets in appropriate spaces, we have 
\begin{equation}\label{show}
P\left[\{E_v\in C\}\cap G\right]=E\left[P[E_v\in C\mid W_v, f(E_v)];G\right].
\end{equation}
The right-hand side of (\ref{show}) is equal to
\begin{align*}
&E\left[P[E_v\in C\cap f^{-1}(D)\mid W_v, f(E_v)]; F_v\in B, W_v\in A\right]\\
=&E\left[E\left[P[E_v\in C\cap f^{-1}(D)\mid W_v, f(E_v)]; F_v\in B, W_v\in A\mid W_v, f(E_v)\right]\right]\\
=&E\left[P[E_v\in C\cap f^{-1}(D)\mid W_v, f(E_v)]\,P\left[F_v\in B\mid W_v, f(E_v)\right]; W_v\in A\right].
\end{align*}
It follows from independence in (EX) and \cite[Proposition 6.6]{Kal02} that a.s.\
$P[F_v\in B\,|\,W_v,f(E_v)]=P[F_v\in B\,|\,W_v]$. Using this fact and
conditioning on $W_v$ inside the above expectation we get that the last
expression above is equal to
\begin{align*}
&E\left[P\left[E_v\in C\cap f^{-1}(D)\mid W_v\right]P\left[F_v\in B\mid W_v\right]; W_v\in A\right]\\
\stackrel{\rm (EX)}{=}&E\left[P\left[F_v\in B, E_v\in C\cap f^{-1}(D)\mid W_v\right]; W_v\in A\right]\\
=&P[ W_v\in A, F_v\in B,E_v\in C\cap f^{-1}(D)],
\end{align*}
which is the same as the left-hand side of (\ref{show}). This proves (\ref{srw}). 

Now let $j\ge 0$ and $y\in V$. Applying (\ref{srw}) to  
$f((s_k)_{k\ge 0}):= (s_k)_{0\le k\le j}$ and $\{E_v\in C\}=\{S_{j+1}(v,0)=y\}$ and subtracting 
$\varkappa(S_j(v,0),y)$ from both sides we obtain
\begin{eqnarray*}
&&P[S_{j+1}(v,0)=y\mid \cal F_j(v,0)]-\varkappa(S_j(v,0),y)\\ 
&=&P\left[S_{j+1}(v,0)=y\mid \si(\cal F_0(v,1)\cup\si(S_k(v,0); 0\le k\le j))\right] 
-\varkappa(S_j(v,0),y)\\  
&\overset d=&P\left[S_{j+1}(v,1)=y\mid \cal F_j(v,1)\right]
-\varkappa(S_j(v,1),y)\ \stackrel{\rm (EL)}{\ge}\ 0\quad\mbox{a.s.}
\end{eqnarray*}
since $(W_v,F_v,E_v)\overset{d}{=}(W_v,E_v,F_v)$ due to (EX). The claim of the lemma follows.
\end{proof}
\vspace*{2mm}

{\bf Acknowledgment:} E. Kosygina was partially
supported by the Simons Foundation through Collaboration Grant for
Mathematicians \# 209493 and Simons Fellowship in Mathematics for 2014-2015. M.\ Zerner was partially supported by the
European Research Council, Grant 208417-NCIRW. This research was supported through the program ``Research in Pairs'' by the Ma\-the\-ma\-ti\-sches Forschungsinstitut Oberwolfach in 2015.

\vspace*{2mm}

{\sc \small
\begin{tabular}{ll}
Department of Mathematics& \hspace*{15mm}Mathematisches Institut\\
Baruch College, Box B6-230& \hspace*{15mm}Universit\"at T\"ubingen\\
One Bernard Baruch Way&\hspace*{15mm}Auf der Morgenstelle 10\\
New York, NY 10010, USA&\hspace*{15mm}72076 T\"ubingen, Germany\\
{\verb+elena.kosygina@baruch.cuny.edu+}& \hspace*{15mm}{\verb+martin.zerner@uni-tuebingen.de+}
\end{tabular}
}

\end{document}